\documentclass[reqno]{amsart}
\usepackage{lmodern}
\usepackage[T5]{fontenc}
\usepackage{amssymb}
\usepackage{amsmath}
\usepackage{amsmath}
\usepackage{bbm}
\usepackage[v2,cmtip]{xy}
\usepackage{mathrsfs}
\usepackage{rotating}
\usepackage{hyperref}
\usepackage{todonotes}
\input diagxy

\newcommand{\colim}{\displaystyle{\lim_{\longrightarrow}}}

\def\f2{{\mathbb F}_{2}}
\def\z2{{\mathbb Z}/2}

\def\fp{{\mathbb F}_{p}}

\DeclareMathOperator{\p}{\mathcal{P}}

\DeclareMathOperator{\im}{im}
\DeclareMathOperator{\Hom}{Hom}
\DeclareMathOperator{\Extgroup}{Ext}
\DeclareMathOperator{\Torgroup}{Tor}
\newcommand{\Ext}[4]{\Extgroup_{{\scriptscriptstyle{#1}} } ^{#2}({#3}, {#4}) }
\newcommand{\ext}[2]{\Extgroup_{{\scriptscriptstyle{#1}} } ^{#2} }
\newcommand{\Tor}[4]{\Torgroup^{{\scriptscriptstyle{#1}} } _{#2}({#3},
  {#4}) }

\newcommand{\tsor}[1]{\otimes_{\scriptscriptstyle{#1}}}

\newtheorem{theorem}{Theorem}[section]
\newtheorem{lemma}[theorem]{Lemma}
\newtheorem{proposition}[theorem]{Proposition}
\newtheorem{corollary}[theorem]{Corollary}
\newtheorem{conjecture}[theorem]{Conjecture}

\newtheorem{remark}[theorem]{Remark}

\numberwithin{equation}{section}

\begin{document}
\title[The Ext group $\Ext{A}{s}{\widetilde{H}^*(B\mathbb{Z}/p)}{\fp}$ and the Lannes-Zarati homomorphism]{The cohomology of the Steenrod algebra and the mod $p$ Lannes-Zarati homomorphism}
\thanks{This research is funded by the National Foundation for Science and Technology Development
(NAFOSTED) of Vietnam.}
\author[P. H. Ch\horn{o}n and P. B. Nh\horn{u}]{Phan Ho{\`a}ng Ch\horn{o}n and Ph\d am B\'ich Nh\horn{u}}
\address{Department of Mathematics and Application, Saigon University, 273 An Duong Vuong, District 5, Ho Chi Minh city, Vietnam.}
\email{phchon@sgu.edu.vn}
\address{Department of Mathematics, College of Science, Cantho University, 3/2 Street, Cantho city, Vietnam.}
\email{pbnhu@ctu.edu.vn}

\date{\today}
\maketitle
\begin{abstract}
In this paper, we compute $\Ext{A}{s}{\widetilde{H}^*(B\mathbb{Z}/p)}{\fp}$ for $s\leq 1$. Using this result, we investigate the behavior of $\varphi_3^{\fp}$ and $\varphi_s^{\widetilde{H}^*(B\mathbb{Z}/p)}\ (s\leq1)$ for an odd prime $p$.
\end{abstract}

\section{Introduction}
The paper is a continuation of our previous one \cite{Chon_Nhu2019}, which will refer to as Part I. In Part I, we construct a chain-level representation of the dual of the mod $p$ Lannes-Zarati homomorphism $(\varphi_s^M)^\#$ in the Singer-Hưng-Sum chain complex \cite{Hung.Sum1995} as well as a chain-level representation of $\varphi_s^{\fp}$ in the lambda algebra \cite{Bousfield.etal1966}. Using the latter to investigate the mod $p$ Lannes-Zarati homomorphism $\varphi_s^M$, we obtain new results for the kernel and the image of $\varphi_s^{\fp}, s\leq 2,$ for $p$ odd. For any unstable $A$-module, let us recall that the mod $p$ Lannes-Zarati homomorphism, which is first intruduced by Lannes and Zarati in \cite{Lan-Zar87}, for each $s\geq0$, is defined as follows
\begin{equation}\label{eq:intro_LZ}
\varphi_s^M:\Ext{A}{s,s+t}{M}{\fp}\to (\fp\tsor{A}\mathscr{R}_sM)^\#_t.
\end{equation}
Where $A$ denote the mod $p$ Steenrod algebra and $\mathscr{R}_sM$ denote the Singer construction (see Singer \cite{Singer78}, \cite{Singer1983}, Lannes-Zarati \cite{Lan-Zar87}, Zarati \cite{Zarati.thesis}, see also H\h{a}i \cite{Hai.thesis}, Powell \cite{Powell2014} and citations therein for detail description).

Our interests in the map \eqref{eq:intro_LZ} (or its dual) lies in the fact that it  is closely related to the mod $p$ Hurewicz map. Indeed, if $M$ is the reduced mod $p$ (singular) cohomology of a pointed space $X$, then $\varphi_s^M$ is considered as a graded associated version of the mod $p$ Hurewicz map
\[
h_*:\pi_*^SX=\pi_*QX\to H_*QX
\]
of the infinite loop space $QX:=\colim\Omega^n\Sigma^n X$ in the $E_2$-term of Adams spectral sequence (see Lannes and Zarati \cite{Lannes1988}, \cite{Lan-Zar1983} for $p=2$ and Kuhn \cite{Kuhn2018} for an odd prime $p$). Hence, the study of the behavior of the mod $p$ Lannes-Zarati homomorphism actually corresponds to the description of the image of the mod $p$ Hurewicz map, and therefore, it also closely corresponds to the famous conjectures on spherical classes of Curtis and Wellington for $S^0$ (see Curtis \cite{Cur75}, Wellington \cite{Well82}) and of Eccles for any pointed CW-complex $X$ (which is menioned in \cite{Zare2009}).

We refer to the introduction of Part I for a detail survey of known facts about the mod $p$ Lannes-Zarati homomorphism.

In Part I, we initiated use the lambda algebra in study of the image and the kernel of the map \eqref{eq:intro_LZ} for $M=\fp$. The advantages of our method is that we can avoid using the knowledge of the so-called ``hit problem'' for $\mathscr{R}_s\fp$ as in \cite{Hung.Peterson1995}, \cite{Hung97}, \cite{Hung2003}, \cite{Hung.et.al2014}. Therefore, it can help us to not only recover previous known results with little computation involved (however, we have not pointed out this fact in Part I), but also obtain new results about the behavior of $\varphi_s^{\fp}$ for $s\leq 2$ with $p$ odd. However, for $s$ higher, the computation remains difficult because the adem relations of the mod $p$ Dyer-Lashof algebra, considered as the dual of $\mathscr{R}_s\fp$, in general, is hard to exploit (see the proof of Theorem 4.2 in \cite{Chon_Nhu2019}).

To overcome this difficulty, in this paper, we develop the power operation $\p^0$ acting on $\Ext{A}{s}{\fp}{\fp}$ (see Liulevicius \cite{Liulevicius1960} or May \cite{Coh-Lad-May76}). For $M=\fp$ and $M=\widetilde{H}^*(B\mathbb{Z}/p)$, we show that there exist the power operations $\p^0$s acting on $\Ext{A}{s}{M}{\fp}$ and on $(\fp\tsor{A}\mathscr{R}_sM)^\#$. Moreover, these actions are compatible with each other through the mod $p$ Lannes-Zarati homomorphism $\varphi_s^M$ (see Proposition \ref{pro:power operations}).

Using the construction above, we have the following, which is the first main our results.

\vspace{.2cm}

\noindent\textbf{Theorem \ref{thm:rank3}.}
\textit{The third Lannes-Zarati homomorphism
\[
\varphi_3^{\fp}:\Ext{A}{3,3+t}{\fp}{\fp}\to (\fp\tsor{A}\mathscr{R}_3\fp)^\#_t
\]
is a monomorphism for $t=0$ and vanishing at all positive stems $t$.}
\vspace{.2cm}

In order to investigate the behavior of $\varphi_s^P$ for $P:=\widetilde{H}^*(B\mathbb{Z}/p)$, we, on one hand, construct a spectral sequence, which is a generalized version of one used in Cohen-Lin-Mahowld \cite{Cohen_Lin_Mahowald1988}, Lin \cite{Lin08} and Chen \cite{Chen2011}. Using this spectral sequence to compute $\Ext{A}{s}{P}{\fp}$, we obtain the following theorems, which are the second main our results.

\vspace{.2cm}

\noindent\textbf{Theorem \ref{thm:Ext(P)^0}} (Cf. \cite[Theorem 1.1]{Crossley1996})\textbf{.}
\textit{The Ext group $\Ext{A}{0,t}{P}{\fp}$ has an $\fp$-basis consisting of all elements
\begin{enumerate}
\item $\widehat{h}_i\in \Ext{A}{0,2(p-1)p^i-1}{P}{\fp}, i\geq0$;
\item $\widehat{h}_i(k)\in \Ext{A}{0,2kp^i-1}{P}{\fp}, i\geq0,1\leq k<p-1$.
\end{enumerate}
}
\vspace{.2cm}

The Ext group $\Ext{A}{1,1+t}{P}{\fp}$ is given by the following theorem.

\vspace{.2cm}

\noindent\textbf{Theorem \ref{thm: Ext(P)^1}.}
\textit{The Ext group $\Ext{A}{1,1+t}{P}{\fp}$ has an $\fp$-basis consisting of all elements given by the following list
\begin{enumerate}
\item $\alpha_0\widehat{h}_i\in \Ext{A}{1,2(p-1)p^i}{P}{\fp},i\geq1$;
\item $\alpha_0\widehat{h}_{i}(k)\in \Ext{A}{1,2kp^i}{P}{\fp},i\geq1,1\leq k<p-1$;
\item $\widehat{\alpha}(\ell)\in \Ext{A}{1,2(p+\ell)+2}{P}{\fp},0\leq\ell<p-2$;
\item $h_i\widehat{h}_i(1)\in \Ext{A}{1,2(p-1)p^i+2p^i-1}{P}{\fp}, i\geq0$;
\item $h_i\widehat{h}_j\in \Ext{A}{1,2(p-1)(p^i+p^j)-1}{P}{\fp},i,j\geq0, j\neq i, i+1$;
\item $h_i\widehat{h}_{j}(k)\in \Ext{A}{1,2(p-1)p^i+2kp^i-1}{P}{\fp}, i,j\geq0, j\neq i, i+1,1\leq r<p-1$;
\item $\widehat{d}_i(k)\in \Ext{A}{1,2(p-1)(p^i+p^{i-1})+2kp^i-1}{P}{\fp},i\geq1,1\leq k\leq p-1$;
\item $\widehat{k}_i(k)\in \Ext{A}{1,2(k+1)p^{i+1}-1}{P}{\fp},i\geq0,1\leq k<p-1$;
\item $\widehat{p}_i(r)\in \Ext{A}{1,2(p-1)(p^i+p^{i+1})+2(r+1)p^i-1}{P}{\fp},i\geq0$, $1\leq r<p-1$.
\end{enumerate}}

\textit{The decomposable elements in $\Ext{A}{s,s+t}{P}{\fp}$ for $s\leq 1$ satisfy only the following relations:
\begin{itemize}
\item $h_i\widehat{h}_{i+1}=0, i\geq0$;
\item $h_i\widehat{h}_{i+1}(k)=0, i\geq0, 1\leq k<p-1$;
\item $h_i\widehat{h}_i=0$, $i\geq0$;
\item $h_i\widehat{h}_i(k)=0$, $i\geq0, 2\leq k<p-1$;
\item $\alpha_0\widehat{h}_0=0$; and
\item $\alpha_0\widehat{h}_0(k)=0,1\leq k<p-1$.
\end{itemize}
}
\vspace{.2cm}

On the other hand, we describe the dual of $\mathscr{R}_sM$ in term of the mod $p$ Dyer-Lashof algebra $R$, which is a quotient algebra of the lambda algebra. We show that $(\mathscr{R}_sM)^\#$ is considered as a quotient right $A$-module of $R_s\otimes M^\#$. Here $R_s$ denote the subspace of $R$ spanned by all monomials of length $s$, and the right $A$-action on $R$ is given via the Nishida relations. Basing on the description, we obtain a chain-level representation of $\varphi_s^M$ for any unstable $A$-module $M$ in term of the lambda algebra.
Using the knowledge above to determine the image and the kernel of $\varphi_s^P$, we also obtain the following results, which are the third main our results.

\vspace{.2cm}
\noindent\textbf{Theorem \ref{thm:varphi_0^P}.}
\textit{The Lannes-Zarati homomorphism $\varphi_0^{P}:\Ext{A}{0,t}{P}{\fp}\to (\fp\tsor{A}\mathscr{R}_0P)^\#_t$ is an isomorphism. 
}
\vspace{.2cm}

This result is similar to the case $p=2$ due to Hưng and Tuấn \cite{Hung-Tuan-preprint}.

\vspace{.2cm}
\noindent\textbf{Theorem \ref{thm:varphi_1^P}.}
\textit{The Lannes-Zarati homomorphism $\varphi_1^{P}:\Ext{A}{1,1+t}{P}{\fp}\to (\fp\tsor{A}\mathscr{R}_1P)^\#_t$ sends
\begin{enumerate}
\item $h_i\widehat{h}_i(1)$ to $\left[\beta Q^{p^i}ab^{[p^i-1]}\right]$, for $i\geq0$;
\item $h_i\widehat{h}_j$ to $\left[\beta Q^{p^i}ab^{[(p-1)p^j-1]}\right]$ for $0\leq j<i$;
\item $h_i\widehat{h}_j(k)$ to $\left[\beta Q^{p^i}ab^{[kp^j-1]}\right]$ for $0\leq j<i$, $1\leq k<p-1$;
\item $\widehat{k}_i(k)$ to $(\p^0)^i\left(
\left[\beta Q^{k+1}ab^{[k]}\right]\right),i\geq0,1\leq k<p-1$; and
\item others to zero.
\end{enumerate}
}
\vspace{.2cm}

Here, we denote $a^\epsilon b^{[s]}$ the $\fp$-generator of $P^\#=\widetilde{H}_*(B\mathbb{Z}/p)$, which is the dual of $x^\epsilon y^s\in P=\widetilde{H}^*(B\mathbb{Z}/p)$ and $Q^i, \beta Q^i$ denote the generators of the mod $p$ Dyer-Lashof algebra.

It is clear that $[\beta Q^{p-1}b^{[1]}+Q^{p-1}a]$ is non-trivial in $(\fp\tsor{A}\mathscr{R}_1P)^\#$ (see Remark \ref{rm:varphi_1P}). It follows that, basing on Theorem \ref{thm:varphi_1^P}, $\varphi_1^P$ is not an epimorphism. This fact is similar to the case $p=2$ (see Remark \ref{rm:vaphi_1^P for p=2}).

It should be note that our strategy in term of the lambda algebra is also valid for the case $p=2$ with a bit modification. Therefore, using this method, we can review the results of Lannes-Zarati \cite{Lan-Zar87} and Hưng el. al.  \cite{Hung.Peterson1995}, \cite{Hung97}, \cite{Hung2003}, \cite{Hung.et.al2014} \cite{Hung-Tuan-preprint} with a little computation (see Appendix B).

From the results in \cite{Chon_Nhu2019} and Theorem \ref{thm:rank3}, we observe that, for an odd prime $p$, the behavior of the mod $p$ Lannes-Zarati $\varphi_s^M, s>2,$ is similar to the case $p=2$. Basing on the conjecture on spherical classes due to Wellington \cite{Well82} and Conjecture 1.2 in \cite{Hung-Tuan-preprint}, it leads us to a conjecture
\begin{conjecture}[{Cf. \cite[Conjecture 1.2]{Hung-Tuan-preprint}}]\label{con:LZ}
Given an unstable $A$-module $M$, the mod $p$ Lannes-Zarati homomorphism
\[
\varphi_s^M:\Ext{A}{s,s+t}{M}{\fp}\to (\fp\tsor{A}\mathscr{R}_sM)^\#_t
\]
is trivial at all positive stems $t$, for $s>2$.
\end{conjecture}

For $p=2$, the conjecture is verified for $M=\f2$ with $3\leq s\leq 5$ by Hưng et. al. (see \cite{Hung.Peterson1995}, \cite{Hung97}, \cite{Hung2003}, \cite{Hung.et.al2014}), and for $M=\widetilde{H}^*(B\mathbb{Z}/2)$ with $3\leq s\leq 4$ by Hưng-Tuấn \cite{Hung-Tuan-preprint}. For $p$ odd and $M=\fp$, Theorem \ref{thm:rank3} shows that this conjecture is true for $s=3$.

For any $A$-module $M$, the chain complex $\Lambda\otimes M^\#$ is a suitable resolution to compute $\Ext{A}{s}{M}{\fp}$, where $\Lambda$ denote the lambda algebra that is isomorphic to the co-Koszul resolution of the mod $p$ Steenrod algebra (see Priddy \cite{Priddy1970}). In addition, for any unstable $A$-module $M$, the dual of Singer construction $(\mathscr{R}_sM)^\#$ can be considered as a quotient module of $\Lambda\otimes M^\#$ (see Proposition \ref{pro:dual of Singer functor}). Observe from the results of Lin \cite{Lin08}, Chen \cite{Chen2011} and Aikawa \cite{Aikawa1980} that all known elements in $\Ext{A}{s,s+t}{\fp}{\fp}$ for $s>2$ can be represented in the lambda algebra by a cycle whose image is trivial in the mod $p$ Dyer-Lashof algebra under the canonical projection. Therefore, it leads us to a conjecture that
\begin{conjecture}[{Cf. \cite[Conjecture 1.4]{Hung-Tuan-preprint}}]\label{con:cycle in lambda}
Given an unstable $A$-module $M$, for $s>2$, every elements of positive stems in $\Ext{A}{s,s+t}{M}{\fp}$ can be represented by a cycle in $\Lambda\otimes M^\#$ whose image is trivial under the canonical projection $\Lambda\otimes M^\#\to (\mathscr{R}_sM)^\#$.
\end{conjecture}

It is clear that Conjecture \ref{con:LZ} is a consequence of Conjecture \ref{con:cycle in lambda}.

The algebraic Singer transfer, which is first constructed by Singer (for $p=2$) \cite{Singer1989}, is defined by, for each $A$-module $M$ and for each integer $s\geq0$,
\[
\psi_s^M:(\fp\tsor{A} (P_s\otimes M))^\#\to\Ext{A}{s}{M}{\fp},
\]
where $P_s=H^*(B(\mathbb{Z}/p)^s)$.
Later, it is generalized by Crossley \cite{Crossley1999} for $p$ odd. Here, we show that (up to sign) the canonical inclusion from $\mathscr{R}_sM$ to $P_s\otimes M$ is a chain-level representation of the following map, for any unstable $A$-module $M$ (see Corollary \ref{cor:representation of j_s}),
\[
j_s^M:=(\psi_s^M)^\#\circ(\varphi_s^M)^\#:\fp\tsor{A}\mathscr{R}_sM\to \fp\tsor{A}(P_s\otimes M).
\]

The Conjecture \ref{con:LZ} leads us to a weaker conjecture as follows
\begin{conjecture}[{Cf. \cite[Conjecture 1.6]{Hung-Tuan-preprint}}]
For any unstable $A$-module $M$, the positive component of the Singer construction $\mathscr{R}_sM$ contains in $\bar{A}(P_s\otimes M)$ for $s>2$, where $\bar{A}$ is the augmentation ideal of $A$.
\end{conjecture}

The conjecture is proved by Hưng-Nam \cite{Hung.Nam2001} for $M=\f2$ and by Hưng-Powell \cite{Hung-Powell2019} for any unstable $A$-module $M$ with $p=2$. However, for $p$ odd, it is still open.

The paper is organized as follows. Section 2 is a preliminary on the Singer-H\horn{u}ng-Sum chain complex, the lambda algebra as well as the Dyer-Lashof algebra that are required for other sections. In Section 3, we construct a spectral sequence to compute $\Ext{A}{s}{P}{\fp}$. Using this spectral sequence, we calculate the ${\rm Ext}$ groups $\Ext{A}{s}{\widetilde{H}^*(B\mathbb{Z}/p)}{\fp}$ for $s\leq 1$. In section 4, we recall the mod $p$ Lannes-Zarati homomorphism and its chain-level representation in Singer-Hưng-Sum chain complex, which is presented in \cite{Chon_Nhu2019}. In addition, we also describe therein the dual of the Singer construction $\mathscr{R}_sM$ in term of the Dyer-Lashof algebra. The section 5 provides a development of the power operations. The behavior of the mod $p$ Lannes-Zarati homomorphism is presented in the final section. In appendix, we construct a chain-level representation of the algebraic Singer transfer and get a description of the map $j_s^M$. In addition, we make some changes needed for the case $p=2$ and recover all known results of the mod $2$ Lannes-Zarati homomorphism.
\section{Preliminaries}
Unless stated otherwise, we will be working over the prime order field $\fp$, where $p$ is an odd prime. The Steenrod algebra over a fixed odd prime $p$ is denoted by $A$. Let $\mathcal{M}$ denote the category of graded left $A$-modules and degree zero $A$-linear map. A module $M\in\mathcal{M}$ is called unstable if $\beta^\epsilon P^ix=0$ for $\epsilon+2i>{\rm deg}(x)$ and for all $x\in M$. The full subcategory of $\mathcal{M}$  of all unstable modules is denoted by $\mathcal{U}$. Given an $A$-module $M$ and an integer $s$, let $\Sigma^s M$ denote the $s$-th iterated suspension of $M$. By definition,   $(\Sigma^s M)^n=M^{n-s}$, thus an element in degree $n$ of $\Sigma^s M$ is usually written in the form $\Sigma^s m$, where $m \in M^{n-s}$. The action of the Steenrod algebra is given by $\theta(\Sigma^s m)=(-1)^{s\deg\theta}\Sigma^s(\theta m)$, for $\theta\in A$ and $m\in M$.
For any left $A$-module $M$, the linear dual of $M$ is denoted by $M^\#$, which admits a right $A$-module structure.

\subsection{The Singer-H\horn{u}ng-Sum chain complex}
Let $E_s$ be an $s$-dimensional $\fp$-vector space.
It is well-known that the mod $p$ cohomology of the classifying space $BE_s$ is given by
\[
P_s:=H^*BE_s=E(x_1,\dots,x_s)\otimes \fp[y_1,\dots,y_s],
\]
where $(x_1,\cdots,x_s)$ is a basis of $H^1BE_s=\Hom(E_s,\fp)$ and $y_i=\beta(x_i)$ for $1\leq i\leq s$ where $\beta$ denote the Bockstein homomorphism. Here $E(\dots)$ and $\fp[\dots]$ are standard notations for the exterior algebra and the polynomial algebra respectively over $\fp$ generated by the indicated variables.

Let $GL_s$ denote the general linear group $GL_s=GL(E_s)$. The group $GL_s$ acts on $E_s$ and then on $H^*BE_s$.
 The algebra of all invariants of $H^*BE_s$ under the actions of $GL_s$ is computed by Dickson \cite{Dic11} and M\`ui \cite{Mui75}. 
We briefly summarize their results. For any $s$-tuple of non-negative integers $(r_1,\dots,r_s)$, put $[r_1,\dots,r_s]:=\det(y_i^{p^{r_j}})$, and define
\[
L_{s,i}:=[0,\dots,\hat{i},\dots,s]; \quad L_{s}:=L_{s,s};\quad q_{s,i}:=L_{s,i}/L_{s},
\]
for any $1 \leq i \leq s$. 

In particular, $q_{s,s}=1$ and by convention, set $q_{s,i}=0$ for $i<0$. The degree of $q_{s,i}$ is $2(p^s-p^i)$. Define
\[
V_s:=V_s(y_1,\dots,y_s):=\prod_{\lambda_j\in\fp}(\lambda_1y_1+\cdots+\lambda_{s-1}y_{s-1}+y_s).
\]

Another way to define $V_s$ is that $V_s=L_{s}/L_{s-1}$.
Then $q_{s,i}$ can be inductively expressed by the formula
\begin{equation}\label{eq:q_s,i in V_i}
q_{s,i}=q_{s-1,i-1}^p+q_{s-1,i}V_s^{p-1}.
\end{equation}

For non-negative integers $k, r_{k+1},\dots,r_s$, set
\[
[k;r_{k+1},\dots,r_s]:=\frac{1}{k!}
\left|
\begin{array}{ccc}
x_1&\cdots&x_s\\
\cdot&\cdots&\cdot\\
x_1&\cdots&x_s\\
y_1^{p^{r_{k+1}}}&\cdots&y_s^{p^{r_{k+1}}}\\
\cdot&\cdots&\cdot\\
y_1^{p^{r_{s}}}&\cdots&y_s^{p^{r_{s}}}
\end{array}
\right|.
\]
For $0\leq i_1<\cdots< i_k\leq s-1$, we define
\begin{align*}
M_{s;i_1,\dots,i_k}&:=[k;0,\dots, \hat{i}_1,\dots,\hat{i}_k,\dots,s-1],\\
R_{s;i_1,\dots,i_k}&:=M_{s;i_1,\dots,i_k}L_{s}^{p-2}.
\end{align*}
From M\`ui \cite{Mui75}, $M_{s;i}$ can be inductively expressed by the formula
\[
M_{s;i}=M_{s-1;i}V_s+q_{s-1,i}M_{s;s-1}.
\]

The subspace of all invariants of $H^*BE_s$ under the action of $GL_s$ is given by the following theorem.

\begin{theorem}[Dickson \cite{Dic11}, M\`ui \cite{Mui75}] \label{thm:invariant of G}
\begin{enumerate}
\item The subspace of all invariants under the action of $GL_{s}$ of $\fp[y_1,\dots,y_s]$  is given by
\[
D[s]:=\fp[y_1,\dots,y_s]^{GL_s}=\fp[q_{s,0},\dots,q_{s,s-1}].
\]
\item As a $D[s]$-module, $(H^*BE_s)^{GL_s}$ is free and has a basis consisting of $1$ and all elements of $\{R_{s;i_1,\dots,i_k}:1\leq k\leq s, 0\leq i_1<\cdots<i_k\leq s-1 \}$.
\item The algebraic relations are given by
\begin{align*}
R_{s;i}^2&=0,\\
R_{s;i_1}\cdots R_{s;i_k}&=(-1)^{k(k-1)/2}R_{s;i_1,\dots,i_k}q_{s,0}^{k-1}
\end{align*}
for $0\leq i_1<\cdots <i_k<s$.
\end{enumerate}
\end{theorem}

Let $\Phi_s:=H^*BE_s[L_s^{-1}]$ be the localization of $H^*BE_s$ obtained by inverting $L_s$. It should be noted that $L_s$ is the product of all non-zero linear forms of $y_1,\dots,y_s$. So inverting $L_s$ is equivalent to inverting all these forms. The action of $GL_s$ on $H^*BE_s$ extends an action of it on $\Phi_s$. Set
\[
\Delta_s:=\Phi_s^{T_s}, \quad \Gamma_s:=\Phi_s^{GL_s},
\]
where $T_s$ is the subgroup of $GL_s$ consisting of all upper triangle matrices with $1$'s on the main diagonal.

Put $u_i:=M_{i;i-1}/L_{i-1}$ and $v_i:=V_i/q_{i-1,0}$, then $|u_i|=1$ and $|v_i|=2$. From \cite{Hung.Sum1995}, we have
\begin{align*}
\Delta_s&=E(u_1,\dots,u_s)\otimes\fp[v_1^{\pm1},\dots,v_s^{\pm1}],\\
\Gamma_s&=E(R_{s;0},\dots,R_{s;s-1})\otimes\fp[q_{s,0}^{\pm1},q_{s,1},\dots,q_{s,s-1}].
\end{align*}

Let $\Delta_s^+$ be the subspace of $\Delta_s$ spanned by all monomials of the form $$u_1^{\epsilon_1}v_1^{(p-1)i_1-\epsilon_1}\cdots u_s^{\epsilon_s}v_s^{(p-1)i_s-\epsilon_s},\epsilon_i\in\{0,1\}, 1\leq i\leq s, i_1\geq\epsilon_1,$$
and let $\Gamma_s^+:=\Gamma_s\cap \Delta_s^+$.

From \cite{Hung.Sum1995}, $\Gamma^+:=\{\Gamma_s^+\}_{s\geq0}$ is a graded differential $\fp$-module with the differential induced by
\begin{equation}\label{eq:diff of gamma}
\begin{split}
\partial(u_1^{\epsilon_1}v_1^{i_1}&\cdots u_s^{\epsilon_s}v_s^{i_s})=
\\
&\left\{
\begin{array}{ll}
(-1)^{\epsilon_1+\cdots+\epsilon_{s-1}}u_1^{\epsilon_1}v_1^{i_1}\cdots u_{s-1}^{\epsilon_{s-1}}v_{s-1}^{i_{s-1}},&\epsilon_s=-i_s=1;\\
0,&\text{otherwise},
\end{array}
\right.
\end{split}
\end{equation}
where $\Gamma_0^+=\fp$.

For any $A$-module $M$, define the stable total power $S_s(x_1,y_1,\dots,x_s,y_s;m)$, for $m\in M$, as follows (see H\horn{u}ng-Sum \cite{Hung.Sum1995})
\begin{align*}
S_s(x_1,y_1,&\dots,x_s,y_s;m):=\\
&\sum_{\tiny \begin{array}{c}\epsilon_j=0,1,\\i_j\geq0\end{array}}(-1)^{\epsilon_1+i_1+\cdots+\epsilon_s+i_s}u_s^{\epsilon_s}\cdots u_1^{\epsilon_1}v_1^{-(p-1)i_1-\epsilon_1}\cdots v_s^{-(p-1)i_s-\epsilon_s}\\
&\quad\quad\quad\quad\quad\quad\quad\quad\quad\quad\otimes (\beta^{\epsilon_1}P^{i_1}\cdots\beta^{\epsilon_s}P^{i_s})(m).
\end{align*}
For convenience, we put $S_s(m):=S_s(x_1,y_1,\dots,x_s,y_s;m)$, and $S_s(M):=\{S_s(m):m\in M\}$.

Then $\Gamma^+M:=\{(\Gamma^+M)_s\}_{s\geq0}$, where $(\Gamma^+M)_0:=M$ and $(\Gamma^+M)_s:=\Gamma_s^+S_s(M)$, is a differential $\fp$-module. For $v=\sum_{\epsilon,\ell}v_{\epsilon,\ell}u_s^\epsilon v_s^{(p-1)\ell-\epsilon}\in\Gamma_s^+$ and $m\in M$, where $v_{\epsilon,\ell}\in \Gamma_{s-1}^+$, the differential in $\Gamma^+M$ is given by
\begin{equation}\label{eq:diff of gamma tensor M}
\partial(v S_s(m))=(-1)^{\deg v+1}\sum_{\epsilon,\ell}(-1)^\ell v_{\epsilon,\ell}S_{s-1}(\beta^{1-\epsilon}P^\ell m).
\end{equation}

In \cite{Hung.Sum1995}, H\horn{u}ng and Sum show that $H_s(\Gamma^+M)\cong \Tor{A}{s}{\fp}{M}$ for any $A$-module $M$. Therefore, $\Gamma^+M$ is a suitable complex to compute $\Tor{A}{s}{\fp}{M}$.
\subsection{The lambda algebra and the Dyer-Lashof algebra}

In \cite{Bousfield.etal1966}, Bousfield et. al. define the lambda algebra (see also Bousfield-Kan \cite{Bousfield-Kan1972}), that is a differential algebra for computing the cohomology of the Steenrod algebra.
H\horn{u}ng and Sum show, in \cite{Hung.Sum1995}, that $\Gamma^+$ is isomorphic to the dual of the lambda algebra as differential $\fp$-modules. However, it is difficult to extend their isomorphim to an isomorphism between chain complexes $\Gamma^+M$ and $\Lambda^\#\otimes M$ because the sign is not compatible. To overcome the difficulty, here we use the opposite algebra of the lambda algebra (see Priddy \cite{Priddy1970}), which corresponds to the original lambda algebra under the anti-isomorphism of differential $\fp$-modules. It is also denoted by $\Lambda$ and called the lambda algebra in the literature.

Recall that $\Lambda$ is the graded, associative, with unite differential algebra over $\fp$ generated by
$\lambda_{i-1}\ (i>0)$ of degree $2i(p-1)-1$ and $\mu_{j-1}\ (j\geq0)$ of degree $2j(p-1)$ satisfying the adem relations (see \cite{Bousfield.etal1966} \cite{Bousfield-Kan1972}, \cite{Well82} and \cite{Priddy1970}) 
{\allowdisplaybreaks
\begin{align*}
\sum_{i+j=n}&\binom{i+j}{i}\lambda_{i-1+pm}\lambda_{j-1+m}=0,\\
\sum_{i+j=n}&\binom{i+j}{i}(\lambda_{i-1+pm}\mu_{j-1+m}-\mu_{i-1+pm}\lambda_{j-1+m})=0,\\
\end{align*}
for all $m\geq1$ and $n\geq 0$; and
\begin{align*}
\sum_{i+j=n}&\binom{i+j}{i}\lambda_{i+pm}\mu_{j-1+m}=0,\\
\sum_{i+j=n}&\binom{i+j}{i}\mu_{i+pm}\mu_{i-1+m}=0,
\end{align*}
}
for all $m\geq0$ and $n\geq 0$.

The differential is given by
\begin{align*}
d(\lambda_{n-1})&=\sum_{i+j=n}\binom{i+j}{i}\lambda_{i-1}\lambda_{j-1},\\
d(\mu_{n-1})&=\sum_{i+j=n}\binom{i+j}{i}(\lambda_{i-1}\mu_{j-1}-\mu_{i-1}\lambda_{j-1}),\\
d(\sigma\tau)&=(-1)^{\deg \sigma}\sigma d(\tau)+d(\sigma)\tau.
\end{align*}

For convenience,  we denote $\lambda_{i-1}^1=\lambda_{i-1}$ and the $\lambda^0_{i-1}=\mu_{i-1}$. Let $\Lambda_s$ denote the subspace of $\Lambda$ spanned by all monomial $\lambda^{\epsilon_1}_{i_1-1}\cdots\lambda^{\epsilon_s}_{i_s-1}$ of the length $s$.
By the adem relations, $\Lambda_s$ has an additive basis consisting of all admissible monomials (which are monomials of the form $\lambda_I=\lambda^{\epsilon_1}_{i_1-1}\cdots\lambda^{\epsilon_s}_{i_s-1}\in\Lambda_s$ satisfying $pi_k-\epsilon_k\geq i_{k-1}$ for $2\leq k\leq s$). 

Given an $A$-module $M$, $\Lambda\otimes M^\#$ is a suitable complex for computing $\Ext{A}{s}{M}{\fp}$, with the differential give by, for $\lambda\in\Lambda$ and $h\in M^\#$,
\begin{align}\label{eq:differential of complex of M}
d(\lambda\otimes h)&=d(\lambda)\otimes h+\sum_{i-\epsilon\geq0}(-1)^{\deg\lambda+(1-\epsilon)\deg h}\lambda\lambda^{\epsilon}_{i-1}\otimes h\beta^{1-\epsilon}P^i.
\end{align}

Using the same method of H\horn{u}ng-Sum \cite{Hung.Sum1995} with a bit modification by multiplying $-1$ to the right hand side of definitions of both operations $\rho$ and $\chi$ (see \cite[Section 5]{Hung.Sum1995}), it is easy to show that the map $\nu=\{\nu_s\}_{s\geq0}$, where $\nu_{s}:\Gamma_s^+\to \Lambda_s^\#$ is given by
\[
\nu_s(u_1^{\epsilon_1}v_1^{(p-1)i_1-\epsilon_1}\cdots u_s^{\epsilon_s}v_s^{(p-1)i_s-\epsilon_s})=(-1)^{i_1+\cdots+i_s+\sum_{\ell<k}\epsilon_\ell\epsilon_k}(\lambda_{i_1-1}^{\epsilon_1}\cdots\lambda_{i_s-1}^{\epsilon_s})^*,
\]
is an isomorphism of differential $\fp$-modules. Moreover, for $A$-module $M$, the map $\nu^M:=\{\nu_s^M\}_{s\geq0}:\Gamma^+M\rightarrow \Lambda^\#\otimes M$ given by $\nu_s^M(vSt_s(m))=\nu_s(v)\otimes m$, for $v\in \Gamma^+_s$ and $m\in M$, is an isomorphism of differential $\fp$-modules.

An important quotient algebra of $\Lambda$ is the mod $p$ Dyer-Lashof algebra $R$, which is also well-known as the algebra of homology operations acting on the homology of infinite loop spaces. 

For any monomial $\lambda_I=\lambda^{\epsilon_1}_{i_1-1}\cdots\lambda^{\epsilon_s}_{i_s-1}\in\Lambda_s$, we define the excess of $\lambda_I$ or of $I$ to be
\[
e(\lambda_I)=e(I)=2i_1-\epsilon_1-\sum_{k=2}^{s}2(p-1)i_k+\sum_{k=2}^s\epsilon_s.
\]
Then, the mod $p$ Dyer-Lashof algebra is the quotient algebra of $\Lambda$ over the (two-side) ideal generated by all monomials of negative excess (see Curtis \cite{Cur75}, Wellington \cite{Well82}).

Let $Q^I=\beta^{\epsilon_1}Q^{i_1}\cdots\beta^{\epsilon_s}Q^{i_s}$ denote the image of $\lambda_I$ under the canonical projection, and
let $R_s$ denote the subspace of $R$ spanned by all monomials of length $s$, then $R_s$ is isomorphic to $\mathscr{B}[s]^\#$ as $A$-coalgebras, where the $A$-action on $R$ is given by the  Nishida's relation (see May \cite{Coh-Lad-May76}).

From the above result, we observe that the restriction of $\nu_s$ on $\mathscr{B}[s]$ is isomorphism between $\mathscr{B}[s]$ and $R_s^\#$.

\section{The cohomology of the Steenrod algebra}\label{sec:cohomology}
In this section, we construct a spectral sequence to compute $\Ext{A}{s}{P}{\fp}$. This spectral sequence is a generalized version of one used in Lin \cite{Lin08} and Chen \cite{Chen2011} for $p$ odd.

Let $H:=\widetilde{H}_*(B\mathbb{Z}/p)=P^\#$. It is known that $H$ has an $\fp$-basis consisting of all elements
\[
\{a^\epsilon b^{[t]}:\epsilon=0,1,s\geq0, t+\epsilon>0\},
\]
where $a,b^{[t]}$ are respectively the dual of $x$ and $y^t$ in $P$. In addition, $H$ admits a right $A$-module structure, with $A$-action given by
\begin{equation}\label{eq:action on H}
b^{[t]}\beta^{\epsilon}P^k=\binom{t-(p-1)k-\epsilon}{k}a^{\epsilon}b^{[t-(p-1)k-\epsilon]},
\end{equation}
and $\theta$ acting trivially on $a$ for $\theta\in \bar{A}$, where $\bar{A}$ is the augmentation ideal of $A$.

Recall that $\Lambda$ denote the lambda algebra. Then $\Lambda\otimes H$ is a suitable complex to compute $\Ext{A}{s}{P}{\fp}$. The differential of $\Lambda\otimes H$ is given by, for $\lambda\in\Lambda$,
\begin{multline}\label{eq:differential of complex of H}
d(\lambda\otimes a^\epsilon b^{[t]})=d(\lambda)\otimes a^\epsilon b^{[t]}+\sum_{i-\delta\geq0}(-1)^{\deg\lambda+(1-\delta)\epsilon}\\\times\binom{t-(p-1)i-1+\delta}{i}\lambda\lambda^{\delta}_{i-1}\otimes a^{\epsilon+1-\delta}b^{[t-(p-1)i-1+\delta]}.
\end{multline}

From Liulevicius \cite{Liulevicius1960}, \cite{Liu62} and May \cite{May70}, there exists a power operation $$\p^0:\Ext{A}{s,s+t}{\fp}{\fp}\to \Ext{A}{s,p(s+t)}{\fp}{\fp}.$$ Its chain-level representation in the lambda algebra is given by
\begin{align*}
\widetilde{\p}^0:\Lambda_s&\to\Lambda_s\\
\lambda^{\epsilon_1}_{i_1-1}\cdots\lambda^{\epsilon_s}_{i_s-1}&\mapsto\left\{
\begin{array}{ll}
\lambda^{\epsilon_1}_{pi_1-1}\cdots\lambda^{\epsilon_s}_{pi_s-1},&\epsilon_1=\cdots=\epsilon_s=1,\\
0,&\text{otherwise.}
\end{array}
\right.
\end{align*}
The operation $\widetilde{\p}^0$ respects the adem relations and commutes with the differential in $\Lambda$.

Define the operation $\theta:H_{t}\to H_{p(t+1)-1}$ given by
\[
\theta(a^\epsilon b^{[t]})=\left\{
\begin{array}{ll}
ab^{[p(t+1)-1]},& \epsilon=1,\\
0,& \epsilon=0,
\end{array}
\right.
\]
which is the dual of the so-called Kameko operation \cite{Kam90} (see also Minami \cite{Min99} for an odd prime $p$).
Since, $(\theta(ab^{[t]}P^i)=(\theta(ab^{[t]}))P^{pi}$, there exists an operation, which is also denoted by $\widetilde{\p}^0$, acting on the chain complex $\Lambda\otimes H$, given as follows
\[
\lambda^{\epsilon_1}_{i_1-1}\cdots\lambda^{\epsilon_s}_{i_s-1}\otimes a^\epsilon b^{[t]}\mapsto\left\{
\begin{array}{ll}
\lambda^{\epsilon_1}_{pi_1-1}\cdots\lambda^{\epsilon_s}_{pi_s-1}\otimes ab^{[p(t+1)-1]},&\epsilon_1=\cdots=\epsilon_s=\epsilon=1,\\
0,&\text{otherwise.}
\end{array}
\right.
\]
The latter commutes with the differential given in \eqref{eq:differential of complex of H}. Therefore, it induces an operation also denoted by $\p^0$ acting on $\Ext{A}{s,*}{P}{\fp}$.

The Ext groups $\Ext{A}{s}{P}{\fp}$ for $s\leq1$ are given by the following theorems.

\begin{theorem}[{Cf. \cite[Theorem 1.1]{Crossley1996}}]\label{thm:Ext(P)^0}
The Ext group $\Ext{A}{0,t}{P}{\fp}$ has an $\fp$-basis consisting of all elements
\begin{enumerate}
\item $\widehat{h}_i:=\left[ab^{[(p-1)p^i-1]}\right]\in \Ext{A}{0,2(p-1)p^i-1}{P}{\fp},i\geq0$; 
\item $\widehat{h}_{i}(k):=\left[ab^{[kp^i-1]}\right]\in \Ext{A}{0,2kp^i-1}{P}{\fp},i\geq0, 1\leq k<p-1$.
 \end{enumerate}
\end{theorem}

\begin{theorem}\label{thm: Ext(P)^1}
The Ext group $\Ext{A}{1,1+t}{P}{\fp}$ has an $\fp$-basis consisting of all elements given by the following list
\begin{enumerate}
\item $\alpha_0\widehat{h}_i=\left[\lambda^0_{-1}ab^{[(p-1)p^i-1]}\right],i\geq1$;
\item $\alpha_0\widehat{h}_{i}(k)=\left[\lambda^0_{-1}ab^{[kp^i-1]}\right],i\geq1,1\leq k<p-1$;
\item $\widehat{\alpha}(\ell)=\left[\lambda^0_{-1}ab^{[p+\ell]}+(\ell+1)\lambda^0_0ab^{\ell+1}\right],0\leq\ell<p-2$;
\item $h_i\widehat{h}_i(1)=\left[\lambda^1_{p^i-1}ab^{[p^i-1]}\right], i\geq0$;
\item $h_i\widehat{h}_j=\left[\lambda^1_{p^i-1}ab^{[(p-1)p^j-1]}\right],i,j\geq0, j\neq i, i+1$;
\item $h_i\widehat{h}_{j}(k)=\left[\lambda^1_{p^i-1}ab^{[kp^j-1]}\right], i,j\geq0, j\neq i, i+1,1\leq r<p-1$;
\item $\widehat{d}_i(k)=(\p^0)^{i-1}\left(\left[\lambda^1_{p-1}ab^{[kp+p-2]}\right]\right),i\geq1,1\leq k\leq p-1$;
\item $\widehat{k}_i(k)=(\p^0)^i\left(
\left[\sum_{j=0}^{k}\frac{1}{j+1}\lambda^{1}_{j}ab^{[(k-j)p+j]}\right]\right),i\geq0,1\leq k<p-1$;
\item $\widehat{p}_i(r)=(\p^0)^i\left(\left[\sum_{j=0}^{p-1-r}\frac{\binom{r+j}{j}}{j+1}\lambda^1_{j}ab^{[(p-j-1)p+r+j]}\right]\right),i\geq0$, $1\leq r<p-1$.
\end{enumerate}

The decomposable elements in $\Ext{A}{s,s+t}{P}{\fp}$ for $s\leq 1$ satisfy only the following relations:
\begin{itemize}
\item $h_i\widehat{h}_{i+1}=0, i\geq0$;
\item $h_i\widehat{h}_{i+1}(k)=0, i\geq0, 1\leq k<p-1$;
\item $h_i\widehat{h}_i=0$, $i\geq0$;
\item $h_i\widehat{h}_i(k)=0$, $i\geq0, 2\leq k<p-1$;
\item $\alpha_0\widehat{h}_0=0$; and
\item $\alpha_0\widehat{h}_0(k)=0,1\leq k<p-1$.
\end{itemize}
\end{theorem}

In order to prove two above theorems, we need to construct a spectral sequence to compute $\Ext{A}{s}{P}{\fp}$ as follows.

We filter the complex $\Lambda\otimes H$ by the filtration $F^n:=F^n(\Lambda\otimes H), n\geq0$ where $F^0(\Lambda\otimes H):=0$ and for $n>0$,

\[
F^n(\Lambda\otimes H):=\{\lambda\otimes h\in\Lambda\otimes H: |h|\leq n\}.
\]
Since \eqref{eq:differential of complex of H}, it is clear that $F^n(\Lambda\otimes H)$ is a subcomplex of $\Lambda\otimes H$ satisfying $\cup_n F^n(\Lambda\otimes H)=\Lambda\otimes H$ and $\cap_n F^n(\Lambda\otimes H)=0$. Therefore, the filtration gives rise to a spectral sequence converging to $\Ext{A}{s}{P}{\fp}$. The differential $d_r:E_r^{n,s,t}\to E_r^{n-r,s+1,t-1}$ is an $\fp$-linear map.

In the spectral sequence, $n$ is the filtration degree, $s$ is the homological degree, $t$ is the internal degree and $s+t$ is the total degree. 

It is easy to see that $E_0^{n,s,t}=(F^n(\Lambda_{s}\otimes H)/F^{n-1}(\Lambda_{s}\otimes H))^t\cong\Sigma^n \Lambda_{s}$, therefore,
\[
E_1^{n,s,t}=H^*(E_0^{n,s,t})\cong \Sigma^n \Ext{A}{s,s+t-n}{\fp}{\fp},
\]
and
 $E_{\infty}^{n,s,t}\cong (F^n H^{s}(\Lambda\otimes H)/F^{n-1}H^{s}(\Lambda\otimes H))^t$, where
\[
F^nH^{s}(\Lambda\otimes H):=\im\left(H^{s}(F^n(\Lambda\otimes H))\to H^{s}(\Lambda\otimes H)\right).
\]
Therefore, $\oplus_{n\geq1}E_{\infty}^{n,s,t}\cong\Ext{A}{s,s+t}{P}{\fp}$.

Basing on the results of Liulevicius \cite{Liu62} (see also Aikawa \cite{Aikawa1980}), we get the $E_1^{*,0,*}$ has a $\fp$-basis consisting of all elements
\begin{itemize}
\item $a^\epsilon b^{[t]}, t+\epsilon\geq1;$
\end{itemize}
the $E_1^{*,1,*}$ has a $\fp$-basis consisting of all elements
\begin{itemize}
\item $\alpha_0 a^\epsilon b^{[t]}=[\lambda^0_{-1}a^\epsilon b^{[t]}], t+\epsilon\geq 1$,
\item $h_i a^\epsilon b^{[t]}=[\lambda^1_{p^i-1}a^\epsilon b^{[t]}], i\geq 0, t+\epsilon\geq 1$;
\end{itemize}
and the $E_1^{*,2,*}$ has an $\fp$-basis consisting of all elements
\begin{itemize}
\item $h_ih_ja^\epsilon b^{[t]}=[\lambda^1_{p^i-1}\lambda^1_{p^j-1}a^\epsilon b^{[t]}],0\leq i<j-1, t+\epsilon\geq1$;
\item $\alpha_0h_ia^\epsilon b^{[t]}=[\lambda^1_{p^i-1}\lambda^0_{-1}a^\epsilon b^{[t]}], i\geq 1, t+\epsilon\geq1$;
\item $\alpha_0^2a^\epsilon b^{[t]}=[(\lambda^0_{-1})^2a^\epsilon b^{[t]}],t+\epsilon\geq1$;
\item $h_{i;2,1}a^\epsilon b^{[t]}=[\lambda^1_{2p^{i+1}-1}\lambda^1_{p^i-1}a^\epsilon b^{[t]}], i\geq0, t+\epsilon\geq 1$;
\item $h_{i;1,2}a^\epsilon b^{[t]}=[\lambda^1_{p^{i+1}-1}\lambda^1_{2p^i-1}a^\epsilon b^{[t]}], i\geq0,t+\epsilon\geq1$;
\item $\rho a^\epsilon b^{[t]}=[\lambda^1_1\lambda^0_{-1}a^\epsilon b^{[t]}],t+\epsilon\geq1$;
\item $\tilde{\lambda}_ia^\epsilon b^{[t]}=\left[\sum_{j=1}^{(p-1)}\frac{(-1)^{j+1}}{j}\lambda^1_{(p-j)p^i-1}\lambda^1_{jp^i-1}a^\epsilon b^{[t]}\right], i\geq0,t+\epsilon\geq 1$.
\end{itemize}

Here we denote $E_r^{*,s,*}=\oplus_{n,t}E_r^{n,s,t}$ and $E_{\infty}^{*,s,*}=\oplus_{n,t}E_{\infty}^{n,s,t}$.

 We will write $\alpha\to\beta$ for meaning that $\alpha$ and $\beta$ survive to $E_r$ and $d_{r}(\alpha)=\beta$ for some $r$, therefore, both $\alpha$ and $\beta$ do not survive to $E_{\infty}^{*,s,*}$. In such case, the element $\beta$ is a boundary, and it is supported by $\alpha$.
 
The differential $E_{r}^{*,0,*}\xrightarrow{d_r}E_{r}^{*,1,*}$ is given by the following lemma.
\begin{lemma}\label{lm:diff of degree 0}
The non-trivial differentials $E_{r}^{*,0,*}\xrightarrow{d_r}E_{r}^{*,1,*}$ are listed as follows:
\begin{enumerate}
\item\label{diff: 0.1} $b^{[t]}\to \alpha_0ab^{[t-1]},$ for $t\geq1$;
\item\label{diff: 0.2} $a b^{[(mp+k)p^i-1]}\to -(k+1) h_iab^{[((m-1)p+k+1)p^i-1]},$ for $i\geq0,1\leq k\leq p-1$, $m\geq1$.
\end{enumerate}
\end{lemma}
\begin{proof}[Proof of Theorem \ref{thm:Ext(P)^0}]
The formula (\ref{diff: 0.2}) of Lemma \ref{lm:diff of degree 0} implies that the element $ab^{[kp^i-1]}$ is an infinite cycle and it survives to $E_{\infty}^{*,0,*}$ for $i\geq0$ and $1\leq k\leq p-1$.
 
 The proof is complete.
\end{proof}

The differential $E_{r}^{*,1,*}\xrightarrow{d_r}E_{r}^{*,2,*}$ is given by the following lemma.
\begin{lemma}\label{lm:diff of degree 1}
The non-trivial differentials $E_{r}^{*,1,*}\xrightarrow{d_r}E_{r}^{*,2,*}$ are listed as follows:
\begin{enumerate}
\item\label{diff:alpha to alpha^2} $\alpha_0b^{[t]}\to \alpha_0^2ab^{[t-1]}$, for $t\geq1$;
\item\label{diff:alpha to rho} $\alpha_0ab^{[mp+\ell]}\to-\binom{\ell+2}{2}\rho ab^{[(m-2)p+\ell+2]}$, for $0\leq \ell<p-2, m\geq2$;
\item\label{diff:alpha to alpha h_i 1} $\alpha_0ab^{[(mp+k)p^i-1]}\to (k+1)\alpha_0h_iab^{[((m-1)p+k+1)p^i-1]},$ for $i\geq1, m\geq1$;
\item\label{diff:alpha to alpha h_i 2} $\alpha_0ab^{[(mp+k)p^i-p+p-2]}\to (k+1)\alpha_0h_iab^{[((m-1)p+k+1)p^i-p+p-2]},$ for $i\geq1,m\geq1$;
\item\label{diff:h_ib to h_ialpha} $h_ib^{[t]}\to -h_i\alpha_0ab^{[t-1]}$, for $i\geq1,t\geq1$;
\item\label{diff:h_0b to rho} $h_0b^{[mp+\ell]}\to \frac{1}{2}(\ell-1)\rho ab^{[(m-1)p+\ell]}$, for $m\geq 1$ and $\ell\neq 1$;
\item\label{diff:h_0b to h_0h_i} $h_0b^{[(mp+r)p^i-p^2+kp+1]}\to-(r+1)h_0h_ib^{[((m-1)p+r+1)p^i-p^2+kp+1]},$ for $i\geq2, m\geq0$;
\item\label{diff:h_iab to h_ih_j, j<i} $h_i ab^{[(mp+k)p^j-1]}\to
-(k+1) h_ih_jab^{[((m-1)p+k+1)p^j-1]}$, for $i\geq1,m\geq1, 0\leq j<i$;
\item\label{diff:h_iab to h_(i-1;1,2)} $h_iab^{[(mp+k)p^{i-1}+p^{i-1}-1]}\to -\binom{k+2}{2}h_{i-1;1,2}ab^{[((m-2)p+k+2)p^{i-1}+p^{i-1}-1]}$, for $i\geq1,m\geq1$;
\item\label{diff:h_iab to h_(i-1;2,1)} $h_iab^{[(mp+k)p^i+(p-1)p^{i-1}-1]}\to -\frac{1}{2}(k-1)h_{i-1;2,1}ab^{[((m-1)p+k)p^i+p^i-1]}$, for $i\geq1, m\geq1$;
\item\label{diff:h_iab to h_(i;;2,1)} $h_iab^{[(mp+k)p^{i+2}+rp^{i+1}+p^i+u]}\to \binom{r+2}{2}h_{i;2,1}ab^{[(mp+k-2)p^{i+2}+(r+2)p^{i+1}+p^i+u]}$, for $i\geq1,m\geq0, u=(p-1)p^{i-1}-1$;
\item\label{diff:h_iab to h_ih_j 1} $h_iab^{[(mp+k)p^{j}-p^{i+2}+v]}\to -(k+1)h_ih_jab^{[((m-1)p+k+1)p^j-p^{i+2}+v]}$, for $j-2\geq i\geq1,m\geq1, v=(p-2)p^{i+1}+p^i+(p-1)p^{i-1}-1$;
\item\label{diff:h_iab to h_ih_j 2} $h_iab^{[(mp+k)p^j-p^{i+1}+p^i+u]}\to -(k+1)h_ih_jab^{[((m-1)p+k+1)p^j-p^{i+1}+p^i+u]}$, for $j-2\geq i\geq 1, m\geq0, u=(p-1)p^{i-1}-1$;
\item\label{diff:h_iab to lambda tilde} $h_iab^{[(mp+k)p^{i+1}+p^i-1]}\to \tilde{\lambda}_i ab^{[((m-1)p+k+2)p^{i+1}-1]}$, for $i\geq0, m\geq 1$;
\item\label{diff:h_iab to h_(i;1,2)} $h_iab^{[mp^{i+2}+kp^{i+1}+rp^i+p^i-1]}\to\frac{(k+2)(r+1)}{2}h_{i;1,2}ab^{[(m-1)p^{i+2}+kp^{i+1}+(r+1)p^i+p^i-1]}$, for $i\geq0,m\geq0$;
\item\label{diff:h_iab to h_ih_j 3} $h_iab^{[(mp+\ell)p^j-p^{i+2}+w]}\to-(\ell+1)h_ih_jab^{[((m-1)p+\ell+1)p^j-p^{i+2}+w]}$, for $j-2\geq i\geq0, m\geq1, w=(p-2)p^{i+1}+rp^i+p^{i}-1$;
\item\label{diff:h_iab to h_ih_j 4} $h_iab^{[(mp+\ell)p^j-p^{i+2}+kp^{i+1}+x]} \to-(\ell+1)h_ih_jab^{[((m-1)p+\ell+1)p^j-p^{i+2}+kp^{i+1}+x]}$, for $j-2\geq i\geq0$, $m\geq1,k\neq p-2,x=p^{i+1}-1$.
\end{enumerate}
\end{lemma}

The proof of Theorem \ref{thm: Ext(P)^1} is proceeded through two steps. In the first step, we compute the $\fp$-basis for $E_{\infty}^{*,1,*}$; and in the second step, we find the representation in the chain complex $\Lambda\otimes H$ of the element of the basis in the first step.

First, the $\fp$-basis of $E_{\infty}^{*,1,*}$ is given by the following proposition.

\begin{proposition}\label{pro:Ext(p)^1}
The infinite term $E_{\infty}^{*,1,*}$ has a $\fp$-basis consisting of all elements given in Table \ref{tab:Ext(P)^1}.
\begin{table}[h]
\caption{The generators of $E_{\infty}^{*,1,*}$}\label{tab:Ext(P)^1}
\begin{tabular}{|c|c|c|}
\hline
Represented by&$t$&Range of indicators\\
\hline
$\alpha_0ab^{[(p-1)p^i-1]}$&$2(p-1)p^i-1$&$i\geq1$\\
\hline
$\alpha_0ab^{[kp^i-1]}$&$2kp^i-1$&$i\geq1,1\leq k<p-1$\\
\hline
$\alpha_0ab^{[p+\ell]}$&$2(p+\ell)+1$&$0\leq\ell< p-2$\\
\hline
$h_iab^{[p^i-1]}$&$2(p-1)p^{i}+2p^i-2$&$i\geq0$\\
\hline
$h_iab^{[(p-1)p^j-1]}$&$2(p-1)(p^i+p^j)-2$&$0\leq j,i; j\neq i, i+1$\\
\hline
$h_iab^{[k p^j-1]}$&$2(p-1)p^i+2k p^j-2$&$0\leq j,i;j\neq i, i+1$\\
&&$1\leq k<p-1$\\
\hline
$h_iab^{[kp^{i}+(p-1)p^{i-1}-1]}$&$2(p-1)(p^i+p^{i-1})+2kp^{i}-2$&$i\geq1,1\leq k\leq p-1$\\
\hline
$h_iab^{[kp^{i+1}+p^{i}-1]}$&$2(k+1)p^{i+1}-2$&$i\geq0,1\leq k<p-1$\\
\hline
$h_iab^{[(p-1)p^{i+1}+(r+1)p^i-1]}$&$2(p-1)(p^i+p^{i+1})+2(r+1)p^i-2$&$i\geq0,1\leq r<p-1$\\
\hline
\end{tabular}
\end{table}
\end{proposition}

\begin{proof}
Since (\ref{lm:diff of degree 1}.\ref{diff:alpha to alpha^2}),
the element $\alpha_0b^{[t]}$ is not an infinite cycle.  Thus, we need only to consider the elements $\alpha_0ab^{[t]}$ for $t\geq0$, $h_ib^{[t]},i\geq0,t>0$ and $h_iab^{[t]}, i\geq0,t\geq0$.
\subsection*{\bf The case $\alpha_0ab^{[t]}$ for $t\geq0$}

From (\ref{diff:alpha to rho}) of Lemma \ref{lm:diff of degree 1}, it is easy to see that $\alpha_0ab^{[\ell]}$ and $\alpha_0ab^{[p+\ell]}$, for $1\leq \ell\leq p-1$, are infinite cycles. The first element is a boundary supported by $b^{[\ell +1]}$. Therefore, we have the following relations:
\[
\alpha_0\widehat{h}_0=0;\quad \alpha_0\widehat{h}_0(k)=0, 1\leq k<p-1.
\]

The second element is also a boundary supported by $b^{[b+\ell+1]}$. However, in $\Lambda\otimes H$, we get
\begin{equation}\label{eq:boundary of alpha ab^p+l}
d(b^{[p+\ell+1]})=\lambda^0_{-1}ab^{[p+\ell]}+(\ell+2)\lambda^1_0b^{[\ell+2]} \mod F^{2(\ell+2)-1}.
\end{equation}

Since, in $\Lambda$, $\lambda^1_0\lambda^0_{-1}=0$, it follows that $\lambda^1_0b^{[\ell+2]}$ is a cycle for $\ell<p-2$. Therefore, $h_0b^{[\ell+2]}$ is an infinite cycle in the spectral sequence, and then, $\alpha_0ab^{[p+\ell]}$ survives to $E_{\infty}^{*,1,*}$ represented by $(\ell+2)h_0b^{[\ell+2]}$ of lower filtration degree.

It is easy to see that, from \eqref{eq:boundary of alpha ab^p+l}, $\alpha_0ab^{[2p-2]}$ does not survives to $E_{\infty}^{*,1,*}$.

For $\ell=p-1$, since $\lambda^0_{-1}ab^{[2p-1]}$ is a cycle in $\Lambda\otimes H$, from \eqref{eq:boundary of alpha ab^p+l}, it implies that $h_ib^{[p+1]}=[\lambda^1_0b^{[p+1]}]$ is an infinite cycle in the spectral sequence. Hence, $\alpha_0ab^{[2p-1]}$ survives to $E_{\infty}^{*,1,*}$.

From (\ref{diff:alpha to rho}) of Lemma \ref{lm:diff of degree 1}, it is sufficient to consider $\alpha_0ab^{[mp+p-1]}$ and $\alpha_0ab^{[mp+p-2]}$ for $m\geq 2$.

In addition, by the formulas (\ref{diff:alpha to alpha h_i 1}) and (\ref{diff:alpha to alpha h_i 2}) of Lemma \ref{lm:diff of degree 1}, we obtain that the element $\alpha_0ab^{[kp^i-1]}$ and $\alpha_0ab^{[kp^i-p+p-2]}$, for $1\leq k\leq p-1$, are infinite cycles.

From (\ref{diff: 0.1}) of Lemma \ref{lm:diff of degree 0}, it is easy to see that $\alpha_0ab^{[kp^i-1]}$ and $\alpha_0ab^{[kp^i-p+p-2]}$ are boundaries and they are respectively supported by $b^{[kp^i]}$ and $b^{[kp^i-1]}$. However, in $\Lambda\otimes H$, we have
\begin{equation}\label{eq:boundary of alpha_0 widehat(h)_i}
d(b^{[kp^i]})=\lambda^0_{-1}ab^{[kp^i-1]}+\lambda^1_0b^{[kp^i-p+1]}\mod F^{2(kp^i-p+1)-1};
\end{equation}
and
\begin{equation*}
d(b^{[kp^i-1]})=\lambda^0_{-1}ab^{[kp^i-p+p-2]}.
\end{equation*}
The second formula implies that the element $\alpha_0ab^{[kp^i-p+p-2]}$ does not survives to $E_{\infty}^{*,1,*}$. Since $\lambda^1_0b^{[kp^i-p+1]}$ represents $h_0b^{[kp^i-p+1]}$ and $\lambda^0_{-1}ab^{[kp^i-1]}$ is a cycle in $\Lambda\otimes H$, it follows that $h_0b^{[kp^i-p+1]}$ is an infinite cycle in the spectral sequence. Therefore, $\alpha_0ab^{[kp^i-1]}$ survives and represents the elements $\alpha_0\widehat{h}_i\neq 0$ and $\alpha_0\widehat{h}_i(k)\neq 0$ for $i\geq 1$ and $1\leq k<p-1$.

Basing above computation, we obtain three first generators in Table \ref{tab:Ext(P)^1}.

\subsection*{\bf The case $h_ib^{[t]}$ for $t>0$}

From (\ref{diff:h_ib to h_ialpha}) of Lemma \ref{lm:diff of degree 1}, it is sufficient to consider $h_0b^{[t]}$ for $t\geq1$. In addition, by (\ref{diff:h_0b to rho}) of Lemma \ref{lm:diff of degree 1}, the element $h_0b^{[\ell]},1\leq \ell\leq p-1$, is an infinite cycle. It is clear that it is in a boundary supported by $b^{[p+\ell-1]}$. However, in $\Lambda\otimes H$, one gets
\begin{equation*}
d\left(b^{[p+\ell-1]}\right)=\lambda^0_{-1}ab^{[p-1+\ell-1]}
+\ell\lambda^1_0b^{[\ell]}\mod F^{2\ell-1}.
\end{equation*}
It follows that $h_0b^{[\ell]}$ survives to $E_{\infty}^{*,1,*}$ and represents the element $\alpha_0ab^{[p+\ell-2]}$.

For $k\geq 1$, by (\ref{diff:h_0b to rho}) of Lemma \ref{lm:diff of degree 1}, it is sufficient to consider the case $\ell=1$.


By (\ref{diff:h_0b to h_0h_i}) of Lemma \ref{lm:diff of degree 1}, the element $h_0b^{[rp^i-p^2+kp+1]}$ is an infinite cycle in the spectral sequence.

For $k=p-1$, basing on \eqref{eq:boundary of alpha_0 widehat(h)_i}, we obtain that $h_0b^{[rp^i-p+1]}$ represents the element $-\alpha_0\widehat{h}_i(r)\neq 0$.

For $k<p-1$, it is easy to see that $h_0b^{[rp^i-p^2+kp+1]}$ is a boundary supported by $b^{[rp^i-p^2+(k+1)p]}$. Furthermore, in $\Lambda\otimes H$, one gets
\begin{multline*}
d\left(b^{[rp^i-p^2+(k+1)p]}\right)=\lambda^0_{-1}ab^{[rp^i-p^2+(k+1)p-1]}\\
+\sum_{j=1}^{p-1}\lambda^1_{j-1}b^{[rp^i-p^2+(k-j+1)p+j]}\\
+(k+2)\lambda^1_{p-1}b^{[rp^i-2p^2+(k+2)p]}\mod F^{2(rp^i-2p^2+(k+2)p)-1}.
\end{multline*}

Since the first term of the right hand side represents $\alpha_0ab^{[rp^i-p^2+(k+1)p-1]}$, which does not survive to $E_{\infty}^{*,1,*}$ (see the first case), it implies that $h_0b^{[rp^i-p^2+kp+1]}$ does not survive to $E_{\infty}^{*,1,*}$.

Thus, this case does not give us any new generator.

\subsection*{\bf The case $h_iab^{[t]}$ for $t\geq0$}

From (\ref{diff:h_iab to h_ih_j, j<i}) of Lemma \ref{lm:diff of degree 1}, it follows that $h_iab^{[kp^j-1]}, 1\leq k\leq p-1,$ is an infinite cycle. It is easy to see that $h_iab^{[kp^j-1]}$ is not a boundary; and, then, it represents the elements $h_i\widehat{h}_j\neq 0$ for $0\leq j<i$ and $h_i\widehat{h}_j(k)\neq 0$ for $0\leq j<i,1\leq k<p-1$.

Since $\binom{k+2}{2}=0$ if and only if $k=p-2$ or $k=p-1$, by (\ref{diff:h_iab to h_(i-1;1,2)}) of Lemma \ref{lm:diff of degree 1}, it is sufficient to consider two cases:
\begin{enumerate}
\item[(a)] $h_iab^{[mp^{i}+(p-1)p^{i-1}-1]}$ for $i\geq1,m\geq1$; and
\item[(b)] $h_iab^{[mp^i+p^i-1]}$ for $i\geq1,m\geq1$.
\end{enumerate}


First, we consider the element (a). By (\ref{diff:h_iab to h_(i-1;2,1)}) of Lemma \ref{lm:diff of degree 1}, for $1\leq k\leq p-1$, the element $h_iab^{[kp^i+(p-1)p^{i-1}-1]}$  is an infinite cycle. 

In addition, it is in a boundary supported by $ab^{[p^{i+1}+(k-1)p^i+(p-1)p^{i-1}-1]}$. However, in $\Lambda\otimes H$, one gets
\begin{multline*}
d\left(ab^{[p^{i+1}+(k-1)p^i+(p-1)p^{i-1}-1]}\right)=
\lambda^1_{p^{i-1}-1}ab^{[p^{i+1}+(k-2)p^i+p^i-1]}\\
-k\lambda^1_{p^i-1}ab^{[kp^i+(p-1)p^{i-1}-1]}\mod F^{2(kp^i+(p-1)p^{i-1}-1)}.
\end{multline*}

Since the second term of the right hand side of the formula is cycles in $\Lambda\otimes H$ and the first term represents the element $h_{i-1}ab^{[p^{i+1}+(k-2)p^i+p^i-1]}$, then it is an infinite cycle. Hence, $h_iab^{[kp^i+(p-1)p^{i-1}-1]}$ survives and represents a non-trivial in $E_{\infty}^{*,1,*}$.


Also by (\ref{diff:h_iab to h_(i-1;2,1)}) of Lemma \ref{lm:diff of degree 1}, for $m\geq1$ the element (a) reduces to the case $k=1$, namely, $h_iab^{[mp^{i+1}+p^i+(p-1)p^{i-1}-1]}$. 

It is clear that if $1\leq m\leq p-1$, then $h_iab^{[mp^{i+1}+p^i+(p-1)p^{i-1}-1]}$ is also a infinite cycle and it is in a boundary supported by $ab^{[(m+1)p^{i+1}+(p-1)p^{i-1}-1]}$. 

In $\Lambda\otimes H$, we get
\begin{multline*}
d\left(ab^{[(m+1)p^{i+1}+(p-1)p^{i-1}-1]}\right)=
\lambda^1_{p^{i-1}-1}ab^{[(m+1)p^{i+1}-1]}\\
-\lambda^1_{p^i-1}ab^{[mp^{i+1}+p^i+(p-1)p^{i-1}-1]}\mod F^{2(mp^{i+1}+p^i+(p-1)p^{i-1}-1)}.
\end{multline*}
It follows that $h_iab^{[mp^{i+1}+p^i+(p-1)p^{i-1}-1]}$ survives to $E_{\infty}^{*,1,*}$, and it represents the element $h_{i-1}\widehat{h}_{i+1}(m+1)\neq0$ for $i\geq1$. It should be noted that when $m=p-2$ and $m=p-1$, the element $\widehat{h}_{i+1}(m+1)$ is respectively equal to $\widehat{h}_{i+1}$ and $\widehat{h}_{i+2}(1)$.


%

When $m\geq p$, put $u=(p-1)p^{i-1}-1$. From (\ref{diff:h_iab to h_(i;;2,1)}) of Lemma \ref{lm:diff of degree 1}, the element (a) reduces to two cases:
\begin{enumerate}
\item[(a.1)] $h_iab^{[(mp+k)p^{i+2}+(p-2)p^{i+1}+p^i+u]}$; and
\item[(a.2)] $h_iab^{[(mp+k)p^{i+2}+(p-1)p^{i+1}+p^i+u]}$.
\end{enumerate}



First, we treat the element (a.1). 

From (\ref{diff:h_iab to h_ih_j 1}) of Lemma \ref{lm:diff of degree 1}, it implies that $h_iab^{[kp^{j}-p^{i+2}+(p-2)p^{i+1}+p^i+u]}$ is an infinite cycle. It is easy to check that, in the spectral sequence, it is in a boundary supported by $ab^{[kp^j-p^{i+1}+(p-1)p^{i-1}-1]}$. In addition, in $\Lambda\otimes H$,
\begin{multline*}
d\left(ab^{[kp^j-p^{i+1}+(p-1)p^{i-1}-1]}\right)=
-\lambda^1_{p^{i-1}-1}ab^{[kp^j-p^{i+1}-1]}\\
-\lambda^1_{p^i-1}ab^{[kp^{j}-p^{i+2}+(p-2)p^{i+1}+p^i+u]}\mod F^{2(kp^{j}-p^{i+2}+(p-2)p^{i+1}+p^i+u)}.
\end{multline*}

Since the first term of the right hand side of the formula represents an element which does not survives to $E_{\infty}^{*,1,*}$, it follows that $h_iab^{[kp^{j}-p^{i+2}+(p-2)p^{i+1}+p^i+u]}$ is a boundary, and then, it does also not survives to $E_{\infty}^{*,1,*}$.


Second, we treat the element (a.2). 

From (\ref{diff:h_iab to h_ih_j 2}) of Lemma \ref{lm:diff of degree 1}, it follows that $h_iab^{[kp^j-p^{i+1}+p^i+(p-1)p^{i-1}-1]}$ is an infinite cycle. It is easy to see that $h_iab^{[kp^j-p^{i+1}+p^i+(p-1)p^{i-1}-1]}$ is in a boundary supported by $ab^{[kp^j+(p-1)p^{i-1}-1]}$. However, in $\Lambda\otimes H$, one gets
\begin{multline*}
d\left(ab^{[kp^j+(p-1)p^{i-1}-1]}\right)=-\lambda^1_{p^{i-1}-1}ab^{[kp^j-1]}\\
+\lambda^1_{p^i-1}ab^{[kp^j-p^{i+1}+p^i+(p-1)p^{i-1}-1]}\mod F^{2(kp^j-p^{i+1}+p^i+(p-1)p^{i-1}-1)}.
\end{multline*}

It follows that, in the spectral sequence, $h_iab^{[kp^j-p^{i+1}+p^i+(p-1)p^{i-1}-1]}$ survives to $E_{\infty}^{*,1,*}$ and represents the element $h_{i-1}\widehat{h}_j(k)\neq 0$, for $j\geq i+2$, $1\leq k\leq p-1$, in $E_{\infty}^{*,1,*}$.


Next, we move on the element (b). It is easy to check that, for $m\leq p-1$, the element $h_iab^{[(m+1)p^i-1]}$ represents the element $h_i\widehat{h}_{i}(m+1)$. However, in $\Lambda\otimes H$, one gets
\begin{equation*}
d\left(ab^{[p^{i+1}+mp^i-1]}\right)=m\lambda^1_{p^i-1}ab^{[(m+1)p^i-1]}\mod F^{2((m+1)p^i-1)}.
\end{equation*}
Therefore, we obtain relations $h_i\widehat{h}_i=0$, $h_i\widehat{h}_i(k)=0$ for $i\geq1$, $2\leq k<p-1$ and $h_i\widehat{h}_i(1)\neq 0$ for $i\geq1$.


From (\ref{diff:h_iab to lambda tilde}) of Lemma \ref{lm:diff of degree 1}, it follows that the element $h_iab^{[(mp+k)p^{i+1}+rp^i+p^i-1]}$ is not an infinite cycle, for $m\geq0, r=0$ and $k=p-1$.

Hence, the element (b) reduces to the following cases:
\begin{enumerate}
\item[(b.1)] $h_iab^{[kp^{i+1}+rp^i+p^i-1]}$ for $k\leq p-2$;
\item[(b.2)] $h_iab^{[(p-1)p^{i+1}+rp^i+p^i-1]}$ for $r>0$;
\item[(b.3)] $h_iab^{[(mp+k)p^{i+1}+rp^i+p^i-1]}$ for $m>0$ and $r>0$.
\end{enumerate}

It should be noted that the two first cases are infinite cycles in the spectral sequence.

By inspection, it is easy to verify that, in $\Lambda\otimes H$,
\begin{equation*}
d\left(ab^{[(k+1)p^{i+1}+(r-1)p^i+p^i-1]}\right)=r\lambda^1_{p^i-1}ab^{[kp^{i+1}+rp^i+p^i-1]}\mod F^{2(kp^{i+1}+rp^i+p^i-1)}.
\end{equation*} 

Therefore, the element (b.1) is boundaries if $r>0$. For $r=p-1$, one gets the relations: $h_i\widehat{h}_{i+1}=0$ and $h_i\widehat{h}_{i+1}(k)=0$ for $1\leq k<p-1$. However, for $r=0$, the element $h_iab^{[kp^{i+1}+p^i-1]}$ survives to $E_{\infty}^{*,1,*}$.

By the same argument, in $\Lambda\otimes H$, one gets
\begin{multline*}
d\left(ab^{[p^{i+2}+(r-1)p^i+p^i-1]}\right)=\sum_{j=1}^{p-r}\binom{r+j-1}{j}\lambda^1_{jp^i-1}ab^{[(p-1)p^{i+1}+(r-1+j)p^i+p^i-1]}\\
+\lambda^1_{p^{i+1}-1}ab^{[p^{i+1}+(r-1)p^i+p^i-1]}\mod F^{2(p^{i+1}+(r-1)p^i+p^i-1)}.
\end{multline*}
Since, for $r>0$, the first sum of the right hand side of the formula is a cycle in $\Lambda\otimes H$ and the last term represents $h_{i+1}ab^{[p^{i+1}+(r-1)p^i+p^i-1]}$, then $h_{i+1}ab^{[p^{i+1}+(r-1)p^i+p^i-1]}$ is an infinite cycle. Therefore, the element (b.2) survives to $E_{\infty}^{*,1,*}$. For $r=p-1$, it represents the element $h_i\widehat{h}_{i+2}(1)\neq 0$.

Finally, basing on (\ref{diff:h_iab to h_(i;1,2)}) of Lemma \ref{lm:diff of degree 1}, the element (b.3) reduces two cases:
\begin{enumerate}
\item[(b.3.1)] $h_iab^{[mp^{i+2}+(p-2)p^{i+1}+rp^i+p^i-1]}$ for $r>0$; and
\item[(b.3.2)] $h_iab^{[mp^{i+2}+kp^{i+1}+p^{i+1}-1]}$ for $k\neq p-2$.
\end{enumerate}


From (\ref{diff:h_iab to h_ih_j 3}) of Lemma \ref{lm:diff of degree 1}, the element $h_iab^{[\ell p^j-p^{i+2}+(p-2)p^{i+1}+rp^i+u]}$, for $r>0, 1\leq \ell\leq p-1$, is an infinite cycle. In addition, it is easy to check that, in $\Lambda\otimes H$,
\begin{multline*}
d\left(ab^{[\ell p^j-p^{i+1}+(r-1)p^i+u]}\right)\\
=r\lambda^1_{p^i-1}ab^{[\ell p^j-p^{i+2}+(p-2)p^{i+1}+rp^i+u]}\mod F^{2(\ell p^j-p^{i+2}+(p-2)p^{i+1}+rp^i+u)}.
\end{multline*}
Since $r>0$, it implies that $h_iab^{[\ell p^j-p^{i+2}+(p-2)p^{i+1}+rp^i+u]}$ does not survive to $E_{\infty}^{*,1,*}$.


Similarly, from (\ref{diff:h_iab to h_ih_j 4}) of Lemma \ref{lm:diff of degree 1}, the element $h_iab^{[\ell p^j-p^{i+2}+kp^{i+1}+p^{i+1}-1]}$ is an infinite cycle. In addition, we also have, in $\Lambda\otimes H$, that
\begin{multline*}
d\left(ab^{[\ell p^j-p^{i+2}+(k+1)p^{i+1}+(p-1)p^i-1]}\right)=-\lambda^1_{p^i-1}ab^{[\ell p^j-p^{i+2}+kp^{i+1}+p^{i+1}-1]}\\
+(k+2)\lambda^1_{p^{i+1}-1}ab^{[\ell p^j-2p^{i+2}+(k+2)p^{i+1}+(p-1)p^i-1]}\mod F^{2u_3},
\end{multline*}
where $u_3=\ell p^j-2p^{i+2}+(k+2)p^{i+1}+(p-1)p^i-1$.

By above computation, $h_{i+1}ab^{[\ell p^j-2p^{i+2}+(k+2)p^{i+1}+(p-1)p^i-1]}$ survives to $E_{\infty}^{*,1,*}$ if and only if $k=p-1$. Therefore, $h_iab^{[\ell p^j-p^{i+2}+kp^{i+1}+p^{i+1}-1]}$ survives to $E_{\infty}^{*,1,*}$ if and only if $k=p-1$. In this case $h_iab^{[\ell p^j-1]}$ represents $h_i\widehat{h}_j(\ell)\neq0$ for $0<i<j-1$ and $1\leq \ell\leq p-1$.

The element $h_0ab^{[t]}$ can be considered as a special case of the element (b) for $i=0$, therefore, it can be treated by the same method.
%

The proof is complete.
\end{proof}
\begin{proof}[Proof of Theorem \ref{thm: Ext(P)^1}]
By direct computation the representation in $\Lambda\otimes H$ of generators given in Table \ref{tab:Ext(P)^1}, we get the first part of the theorem.

The second part of the theorem is followed from the proof of Proposition \ref{pro:Ext(p)^1}. 
\end{proof}

The rest of the section gives a proof of Lemma \ref{lm:diff of degree 1}.
\begin{proof}[Proof of Lemma \ref{lm:diff of degree 1}]
The lemma is proved by direct computation in the chain complex $\Lambda\otimes H$ and in the spectral sequence.

\noindent\textbf{Proof of (\ref{diff:alpha to alpha^2}).} It is clear that, in $\Lambda\otimes H$,
\[
d(\lambda^0_{-1}b^{[t]})=\lambda^0_{-1}\lambda^0_{-1}ab^{[t-1]}\mod F^{2(t-1)},
\]
for $t\geq1$. Therefore, in the spectral sequence, we get the formula.

\noindent\textbf{Proof of (\ref{diff:alpha to rho}).} By inspection, we have, in $\Lambda\otimes H$,
\begin{multline*}
d\left(\lambda^0_{-1}ab^{[mp+\ell]}+(\ell+1)\lambda^0_0ab^{[(m-1)p+\ell+1]}+\binom{\ell+2}{2}\lambda^0_1ab^{[(m-2)p+\ell+2]}\right)=\\
-\binom{\ell+2}{2}\lambda^1_1\lambda^0_{-1}ab^{[(m-2)p+\ell +2]}\mod F^{2((m-2)p+\ell +2)},
\end{multline*}
for $0\leq \ell<p-2, m\geq2$. Therefore, in the spectral sequence, we obtain the formula.

\noindent\textbf{Proof of (\ref{diff:alpha to alpha h_i 1}).} By inspection, in $\Lambda\otimes H$, we have
\begin{equation*}
d(\lambda^0_{-1}ab^{[(mp+k)p^i-1]})=(k+1)\lambda^0_{-1}\lambda^1_{p^i-1}ab^{[((m-1)p+k+1)p^i-1]}\mod F^{2k_1},
\end{equation*}
for $i\geq 1$, where $k_1=((m-1)p+k+1)p^i-1$. Therefore, we get the formula in the spectral sequence.

\noindent\textbf{Proof of (\ref{diff:alpha to alpha h_i 2}).} It is easy to check that, in $\Lambda\otimes H$,
\begin{multline*}
d(\lambda^0_{-1}ab^{[(mp+k)p^i-p+p-2]}+\lambda^0_0ab^{[(mp+k)p^i-p-1]})=\\
(k+1)\lambda^0_{-1}\lambda^1_{p^i-1}ab^{[((m-1)p+k+1)p^i-p+p-2]}\mod F^{2k_2},
\end{multline*}
for $i\geq 1$, where $k_2=((m-1)p+k+1)p^i-p+p-2$. Therefore, in the spectral sequence, we obtain (\ref{diff:alpha to alpha h_i 2}).

\noindent\textbf{Proof of (\ref{diff:h_ib to h_ialpha}).} Similar to the proof of (\ref{diff:alpha to alpha^2}).

\noindent\textbf{Proof of (\ref{diff:h_0b to rho}).} In $\Lambda\otimes H$, we have
\begin{multline*}
d\left(\lambda^1_0b^{[kp+\ell]}+\frac{1}{2}(\ell+1)\lambda^1_1b^{[(k-1)p+\ell+1]}\right)=\\
\frac{1}{2}(\ell-1)\lambda^1_1\lambda^0_{-1}ab^{[(k-1)p+\ell]}\mod F^{2((k-1)p+\ell)},
\end{multline*}
for $k\geq1$ and $\ell\neq 1$. 
Hence, in the spectral sequence, one gets the formula.

\noindent\textbf{Proof of (\ref{diff:h_0b to h_0h_i}).} Put $s=(mp+r)p^i$ for $m\geq0$ and $0\leq r\leq p-1$. By inspection, in $\Lambda\otimes H$, one gets
\begin{multline*}
d\left(\sum_{j=0}^{p-1}\lambda^1_jb^{[s-p^2+(k-j)p+j+1]}+C_1+\cdots+C_{p-k-1}\right)=\\
-(r+1)\lambda^1_0\lambda^1_{p^i-1}b^{[((m-1)p+r+1)p^i-p^2+kp+1]}\\\mod F^{2(((m-1)p+r+1)p^i-p^2+kp+1)-1},
\end{multline*}
where
\[
C_n=\sum_{\ell=0}^{p-1}\binom{p+k+n-\ell}{n}\lambda^1_{np+\ell} b^{[s-(n+1)p^2+(k+n-\ell)p+\ell+1]},
\]
for $i\geq 2$ and $m\geq1$.
Thus, in the spectral sequence, we obtain the formula.

\noindent\textbf{Proof of (\ref{diff:h_iab to h_ih_j, j<i}).} It is immediate.

\noindent\textbf{Proof of (\ref{diff:h_iab to h_(i-1;1,2)}).} By inspection, in $\Lambda\otimes H$, one gets
\begin{multline*}
d\left(\lambda^1_{p^i-1}ab^{[(mp+k)p^{i-1}+p^{i-1}-1]}\right)=
-(k+1)\lambda^1_{p^i-1}\lambda^1_{p^{i-1}-1}ab^{[((m-1)p+k+1)p^{i-1}+p^{i-1}-1]}\\
-\binom{k+2}{2}\lambda^1_{p^i-1}\lambda^1_{2p^{i-1}-1}ab^{[((m-2)p+k+2)p^{i-1}+p^{i-1}-1]}\\\mod F^{2(((m-2)p+k+2)p^{i-1}+p^{i-1}-1)},
\end{multline*}
for $i\geq1$ and $m\geq1$. Since $\lambda^1_{p^i-1}\lambda^1_{p^{i-1}-1}=0$ in $\Lambda$ and the second term of the formula represents $h_{i-1;1,2}ab^{[((m-2)p+k+2)p^{i-1}+p^{i-1}-1]}$, in the spectral sequence, we obtain (\ref{diff:h_iab to h_(i-1;1,2)}) of the lemma.

\noindent\textbf{Proof of (\ref{diff:h_iab to h_(i-1;2,1)}).} In $\Lambda\otimes H$, we have, letting $s=(m-1)p+k$ for $m\geq1$ and $0\leq k\leq p-1$,
\begin{multline*}
d\left(2\lambda^1_{p^i-1}ab^{[(mp+k)p^i+(p-1)p^{i-1}-1]}+(k+1)\lambda^1_{2p^i-1}ab^{[(s+1)p^i+(p-1)p^{i-1}-1]}\right)=\\
((k+1)\lambda^1_{2p^i-1}\lambda^1_{p^{i-1}-1}+2k\lambda^1_{p^i-1}\lambda^1_{p^i+p^{i-1}-1})ab^{[sp^i+p^i-1]}\mod F^{2(sp^i+p^i-1)},
\end{multline*}
for $i\geq1$ and $m\geq1$. 
Since, by adem relation, in $\Lambda$, $\lambda^1_{2p^i-1}\lambda^1_{p^{i-1}-1}+\lambda^1_{p^i-1}\lambda^1_{p^i+p^{i-1}-1}=0$ and the element $\lambda^1_{2p^i-1}\lambda^1_{p^{i-1}-1}ab^{[sp^i+p^i-1]}$ represents the element $h_{i-1;2,1}ab^{[sp^i+p^i-1]}$, in the spectral sequence, we obtain the formula.

\noindent\textbf{Proof of (\ref{diff:h_iab to h_(i;;2,1)}).} Put $u=(p-1)p^{i-1}-1$. In $\Lambda\otimes H$, using adem relation, one gets that the differential of the element

\begin{multline*}
\lambda^1_{p^i-1}ab^{[(mp+k)p^{i+2}+rp^{i+1}+p^i+u]}+\sum_{j=2}^{p-1}\lambda^1_{jp^i-1}ab^{[(mp+k)p^{i+2}+(r-j+1)p^{i+1}+jp^i+u]}\\
+(r+1)\lambda^1_{p^{i+1}+p^i-1}ab^{[(mp+k-1)p^{i+2}+(r+1)p^{i+1}+p^i+u]}\\
+\binom{r+2}{2}\lambda^1_{2p^{i+1}+p^i-1}ab^{[(mp+k-2)p^{i+2}+(r+2)p^{i+1}+p^i+u]}
\end{multline*}
is equal to
\begin{multline*}
\binom{r+2}{2}\lambda^1_{2p^{i+1}-1}\lambda^1_{p^i-1}ab^{[(mp+k-2)p^{i+2}+(r+2)p^{i+1}+p^i+u]}\\
\mod F^{2((mp+k-2)p^{i+2}+(r+2)p^{i+1}+p^i+u)},
\end{multline*}
for $i\geq1$ and $m\geq0$. 

Since, in the spectral sequence, $\lambda^1_{2p^{i+1}-1}\lambda^1_{p^i-1}ab^{[(mp+k-2)p^{i+2}+(r+2)p^{i+1}+p^i+u]}$ represents $h_{i;2,1}ab^{[(mp+k-2)p^{i+2}+(r+2)p^{i+1}+p^i+u]}$, we obtain the formula.

\noindent\textbf{Proof of (\ref{diff:h_iab to h_ih_j 1}).} In $\Lambda\otimes H$, one gets, for $u=(p-1)p^{i-1}-1$ and $v'=(mp+k)p^j$, that the differential of the element
\begin{multline*}
\lambda^1_{p^i-1}ab^{[v'-p^{i+2}+(p-2)p^{i+1}+p^i+u]}+\sum_{\ell=2}^{p-1}\lambda^1_{\ell p^i-1}ab^{[v'-p^{i+2}+(p-1-\ell)p^{i+1}+\ell p^i+u]}\\
-\lambda^1_{p^{i+1}+p^i-1}ab^{[v'-2p^{i+2}+(p-1)p^{i+1}+p^i+u]}
\end{multline*}
is equal to
\begin{equation*}
-(k+1)\lambda^1_{p^i-1}\lambda^1_{p^j-1}ab^{[((m-1)p+k+1)p^j-p^{i+2}+(p-2)p^{i+1}+p^i+u]}\mod F^{2u_1},
\end{equation*}
for $1\leq i\leq j-2$ and $m\geq1$, where $u_1=((m-1)p+k+1)p^j-p^{i+2}+(p-2)p^{i+1}+p^i+u$. Therefore, in the spectral sequence, one obtains the formula.

\noindent\textbf{Proof of (\ref{diff:h_iab to h_ih_j 2}).} In $\Lambda\otimes H$, one gets, for $j-2\geq i\geq 1$ and $u=(p-1)p^{i-1}-1$,
\begin{multline*}
d\left(\lambda^1_{p^i-1}ab^{[v-p^{i+1}+p^i+u]}+\sum_{\ell=2}^{p-1}\lambda^1_{\ell p^i-1}ab^{[v-\ell p^{i+1}+\ell p^i+u]}\right)\\
-(k+1)\lambda^1_{p^i-1}\lambda^1_{p^j-1}ab^{[((m-1)p+k+1)p^j-p^{i+1}+p^i+u]}\mod F^{2u_2},
\end{multline*}
for $m\geq 1$, where $u_2=((m-1)p+k+1)p^j-p^{i+1}+p^i+u$. Hence, in the spectral sequence, we obtain the formula.

\noindent\textbf{Proof of (\ref{diff:h_iab to lambda tilde}).} By inspection, in $\Lambda\otimes H$, we have
\begin{multline*}
d\left(\lambda^1_{p^i-1}ab^{[(mp+k)p^{i+1}+p^i-1]}+\sum_{j=2}^{p-1}\frac{1}{j}\lambda^1_{jp^i-1}ab^{[(mp+k-j+1)p^{i+1}+(j-1)p^i+p^i-1]}\right)\\
=\sum_{j=1}^{p-1}\frac{(-1)^{j+1}}{j}\lambda^1_{(p-j)p^i-1}\lambda^1_{jp^i-1}ab^{[((m-1)p+k+2)p^{i+1}-1]}\mod F^{2(((m-1)p+k+2)p^{i+1}-1)},
\end{multline*}
for $i\geq 0$ and $m\geq 1$. 

Since the right hand side of the latter represents $\tilde{\lambda}_iab^{[((m-1)p+k+2)p^{i+1}-1]}$, we obtain the formula.

\noindent\textbf{Proof of (\ref{diff:h_iab to h_(i;1,2)}).} We have, in $\Lambda\otimes H$, the differential of
\begin{multline*}
\lambda^1_{p^i-1}ab^{[mp^{i+2}+kp^{i+1}+rp^i+p^i-1]}+\sum_{j=2}^{p-r}\frac{\binom{r+j-1}{j-1}}{j}\lambda^1_{jp^i-1}ab^{[mp^{i+2}+(k-j+1)p^{i+1}+(r+j)p^i-1]}\\
+(k+1)\lambda^1_{p^{i+1}+p^i-1}ab^{[(m-1)p^{i+2}+(k+1)p^{i+1}+rp^i+p^i-1]}\\
+\frac{1}{2}k(r+1)\lambda^1_{p^{i+1}+2p^i-1}ab^{[(m-1)p^{i+2}+kp^{i+1}+(r+1)p^i+p^i-1]}
\end{multline*}
is equal to
\begin{multline*}
\frac{1}{2}k(r+1)\lambda^1_{p^{i+1}-1}\lambda^1_{2p^i-1}ab^{[(m-1)p^{i+2}+kp^{i+1}+(r+1)p^i+p^i-1]}\\
-(r+1)\lambda^1_{p^{i+1}+p^i-1}\lambda^1_{p^i-1}ab^{[(m-1)p^{i+2}+kp^{i+1}+(r+1)p^i+p^i-1]}\\
\mod F^{2((m-1)p^{i+2}+kp^{i+1}+(r+1)p^i+p^i-1)},
\end{multline*}
for $i\geq0$ and $m\geq 1$.

Since, by using the adem relation, $\lambda^1_{p^{i+1}-1}\lambda^1_{2p^i-1}+\lambda^1_{p^{i+1}+p^i-1}\lambda^1_{p^i-1}=0$, it follows that, in the spectral sequence, one gets the formula.

\noindent\textbf{Proof of (\ref{diff:h_iab to h_ih_j 3}).} In $\Lambda\otimes H$, for $j-2\geq i\geq0,r>0, m\geq1$ and $u=p^i-1$, the differential of
\begin{multline*}
\lambda^1_{p^i-1}ab^{[(mp+\ell)p^j-p^{i+2}+(p-2)p^{i+1}+rp^i+u]}+\\
\sum_{j=2}^{p-r}\frac{\binom{r+j-1}{j-1}}{j}\lambda^1_{jp^i-1}ab^{[(mp+\ell)p^j-p^{i+2}+(p-1-j)p^{i+1}+(r+j-1)p^i+u]}\\
-\lambda^1_{p^{i+1}+p^i-1}ab^{[(mp+\ell)p^j-2p^{i+2}+(p-1)p^{i+1}+rp^i+u]}\\
-(r+1)\lambda^1_{p^{i+1}+2p^i-1}ab^{[(mp+\ell)p^j-2p^{i+2}+(p-2)p^{i+1}+(r+1)p^i+u]}
\end{multline*}
is equal to
\begin{multline*}
-(\ell+1)\lambda^1_{p^i-1}\lambda^1_{p^j-1}ab^{[((m-1)p+\ell+1)p^j-p^{i+2}+(p-2)p^{i+1}+rp^i+u]}\\
\mod F^{2(((m-1)p+\ell+1)p^j-p^{i+2}+(p-2)p^{i+1}+rp^i+u)}.
\end{multline*}
Therefore, in the spectral sequence, one gets the formula.

\noindent\textbf{Proof of (\ref{diff:h_iab to h_ih_j 4}).} Similarly, in $\Lambda\otimes H$, we also have that, for $j-2\geq i\geq0, m\geq1$ and $k\neq p-2$, the differential of
\begin{multline*}
\lambda^1_{p^i-1}ab^{[(mp+\ell)p^j-p^{i+2}+kp^{i+1}+p^{i+1}-1]}+\\
(k+1)\lambda^1_{p^{i+1}+p^i-1}ab^{[(mp+\ell)p^j-2p^{i+2}+(k+1)p^{i+1}+p^{i+1}-1]}
\end{multline*}
is equal to
\begin{multline*}
-(\ell+1)\lambda^1_{p^i-1}\lambda^1_{p^j-1}ab^{[((m-1)p+\ell+1)p^j-p^{i+2}+kp^{i+1}+p^{i+1}-1]}\\
\mod F^{2(((m-1)p+\ell+1)p^j-p^{i+2}+kp^{i+1}+p^{i+1}-1)}.
\end{multline*}
Hence, in the spectral sequence, we get the formula.
\end{proof}
\section{The mod $p$ Lannes-Zarati homomorphism}\label{sec:LZ homo} 
The destabilization functor $\mathscr{D} \colon \mathcal{M}\rightarrow\mathcal{U}$ is the left adjoint to the inclusion $\mathcal{U} \to \mathcal{M}$. It can be described more explicitly as follows: 
\[
\mathscr{D}(M):=M/EM, 
\]
where $EM:={\rm Span}_{\fp}\{\beta^\epsilon P^ix \colon \epsilon+2i>{\rm deg}(x), x\in M\}$. That $EM$ is an $A$-submodule of $M$ is a consequence of the adem relations. In particular, if $M$ is a graded vector space, considered as an $A$-module with trivial action, then $EM$ is the subspace of elements in negative degrees and, therefore, $\mathscr{D}(M)$ can be identified with the (trivial) $A$-submodule of $M$ consisting of elements in non-negative degrees. This simple observation leads to the following construction. 

For any $A$-module $M$, the projection $M \to\fp\tsor{A}M$ induces an $A$-homomorphism $\mathscr{D}(M)\to \mathscr{D}(\fp\tsor{A}M)$. Thus, there exists a natural $A$-homomorphism $\mathscr{D}(M) \to \fp\tsor{A} M$ which is the composition 
\[
\bfig
\morphism(0,500)<700,0>[\mathscr{D}(M)`\mathscr{D}(\fp\tsor{A}M);]
\morphism(700,500)/^{(}->/<700,0>[\mathscr{D}(\fp\tsor{A}M)`\fp\tsor{A}M.;]
\efig
\]
This in turns induces maps between corresponding derived functors: 
\[
i_s^M \colon \mathscr{D}_s (M) \to \Tor{A}{s}{\fp}{M}. 
\]  

The natural map $i_s^M$ raises the possibility of understanding the homology of the Steenrod algebra via knowledge of derived functors of destabilization. However, it is generally very difficult to compute $\mathscr{D}_s$, except in one important situation in which Lannes and Zarati \cite{Lan-Zar87}, \cite{Zarati.thesis} discover that it can be described in terms of the Singer functors $\mathscr{R}_s$ (see below). 

We proceed to describe Lannes and Zarati discovery. 


Define $\alpha_1(M) \colon \mathscr{D}_{r}(\Sigma^{-1}M)\to \mathscr{D}_{r-1}(P_1\otimes M)$ to be the connecting homomorphism of the functor $\mathscr{D}(-)$ associated to the short exact sequence
\[
0\rightarrow P_1\otimes M\rightarrow \hat{P}\otimes M\rightarrow \Sigma^{-1}M\rightarrow 0,
\]
where $\hat{P}$ is the $A$-module extension of $P_1$ by formally adding a generator $x_1y_1^{-1}$ in degree $-1$. The action of $A$ on $\hat{P}$ is given by declaring $P^n(x_1y_1^{-1})=\binom{-1}{n}x_1y_1^{n(p-1)-1}=(-1)^nx_1y_1^{n(p-1)-1}$ and $\beta(x_1y_1^{-1})=1$, while $A$ acts on $P_1$ in the usual way. It can be verified directly that this gives a well-defined $A$-module structure on $\hat{P}$ which admits $P_1$ as an $A$-submodule.  
 
Put $$\alpha_s(M):=\alpha_1(P_{s-1}\otimes M)\circ\cdots\circ \alpha_1(\
\Sigma^{-(s-1)}M),$$
then $\alpha_s(M)$ is an $A$-linear map from $\mathscr{D}_{r}(\Sigma^{-s}M)$ to $ \mathscr{D}_{r-s}(P_s\otimes M)$. In particular, when $r=s$, we obtain a map $\alpha_s(M) \colon \mathscr{D}_{s}(\Sigma^{-s}M)\to \mathscr{D}_{0}(P_s\otimes M)$. 


On the other hand, for an unstable $A$-module $M$, the Singer  construction $\mathscr{R}_s$ provides a functorial $A$-submodule $\mathscr{R}_s M$ of $P_s\otimes M$. Lannes and Zarati \cite{Lan-Zar87} for $p=2$ and Zarati \cite{Zarati.thesis} for $p$ odd showed that the image of $\alpha_s(\Sigma M):\mathscr{D}_s(\Sigma^{1-s}M)\to \Sigma\mathscr{R}_sM\subset\mathscr{D}_0(P_s\otimes \Sigma M)\cong P_s\otimes \Sigma M$ is an isomorphism.

By the same method in \cite{Lan-Zar87}, \cite{Hung2001}, \cite{Hung-Tuan-preprint} and \cite{Chon_Nhu2019}, for any unstable $A$-module $M$ and for $s\geq0$, there exists a homomorphism $(\bar{\varphi}_s^M)^\#$ such that the following diagram commutes:
\[
\bfig
\ptriangle(0,500)/->`->`.>/<900,500>[\mathscr{D}_s(\Sigma^{1-s}M)`\Sigma\mathscr{R}_sM`\Tor{A}{s}{\fp}{\Sigma^{1-s}M}.;\alpha_s(\Sigma M)`i_s^{\Sigma^{1-s}M}`(\bar{\varphi}_s^M)^\#]
\morphism(900,1000)/^{(}->/<700,0>[\Sigma\mathscr{R}_sM`\Sigma P_s\otimes M;]
\efig
\]


Because the Steenrod algebra $A$ acts trivially on the target, $(\bar{\varphi}_s^M)^\#$ factors through $\fp\tsor{A}\Sigma\mathscr{R}_sM$. Therefore, after suspending $-1$ degree, we obtain the dual of the mod $p$ Lannes-Zarati homomorphism
\[
(\varphi_s^M)^\#:(\fp\tsor{A}\mathscr{R}_sM)^t\rightarrow \Tor{A}{s,t}{\fp}{\Sigma^{-s}M}\cong\Tor{A}{s,s+t}{\fp}{M}.
\]

The linear dual
\[
\varphi_s^M:\Ext{A}{s,t}{M}{\fp}\to (\fp\tsor{A}\mathscr{R}_sM)^\#_t,
\]
is called the mod $p$ Lannes-Zarati homomorphism.



In order to construct a chain-level representation of $\varphi_s^M$, we need to recall some main points of the Singer construction.

Let $\Sigma_{p^s}$ be the symmetric group (of all points) of the group $E_s:=(\mathbb{Z}/p)^s$ and $r_s:E_s\hookrightarrow\Sigma_{p^s}$ be the inclusion via the action by translations. Denote $\mathbb{Z}/p$ the trivial $\Sigma_{p^s}$-module of $\mathbb{Z}/p$ and $\mathcal{Z}/p$ the $\Sigma_{p^s}$-module of $\mathbb{Z}/p$ via the signature action. Put
\[
\mathscr{B}[s]:=\im\left(H^*(B\Sigma_{p^s};\mathbb{Z}/p)\xrightarrow{r_s^*} H^*(BE_s;\mathbb{Z}/p)\right);
\]
\[
\mathcal{B}[s]:=\im\left(H^*(B\Sigma_{p^s};\mathcal{Z}/p)\xrightarrow{r_s^*} H^*(BE_s;r_s^*\mathcal{Z}/p)\right).
\]

The structure of $\mathscr{B}[s]$ and $\mathcal{B}[s]$ are given by the following proposition.
\begin{proposition}[M\`ui \cite{Mui75}, Zarati \cite{Zarati.thesis}]\label{pro:B[s]}
\begin{enumerate}
\item $\mathscr{B}[s]$ is a free $D[s]$-module generated by
\[
\left\{1,M_{s;i_1,\dots,i_k}L_s^{p-2+(p-1)[\frac{k-1}{2}]}\right\}
\]
for $0\leq i_1<\cdots<i_k\leq s-1$.
\item $\mathcal{B}[s]$ is a free $D[s]$-module generated by
\[
\left\{L_s^{\frac{p-1}{2}},M_{s;i_1,\dots,i_k}L_s^{\frac{p-3}{2}+(p-1)[\frac{k}{2}]}\right\}
\]
for $0\leq i_1<\cdots<i_k\leq s-1$.
\end{enumerate}
Here denote $[x]$ the largest integer number that is not greater than $x$.
\end{proposition}

For any unstable $A$-module $M$, the (unstable) total power $St_s(x_1,y_1,\dots,x_s,y_s;m)$, for $m\in M$, is defined as follows (see Zarati \cite{Zarati.thesis})
\[
St_s(x_1,y_1,\dots,x_s,y_s;m):=(-1)^{s\left[\frac{|m|}{2}\right]}L_s^{\frac{p-1}{2}|m|}S_s(m).
\]

For convenience, we put $St_s(m):=St_s(x_1,y_1,\dots,x_s,y_s;m)$ and $St_s(M):=\{St_s(m):m\in M\}$.

Given an unstable $A$-module $M$, the module $\mathscr{R}_sM$ is defined by (see Zarati \cite{Zarati.thesis})
\[
\mathscr{R}_sM=\mathscr{B}[s]\cdot St_s(M^+)\oplus \mathcal{B}[s]\cdot St_s(M^-),
\]
 where $M^+$ (resp. $M^-$) is the subspace consisting of all elements of even degree (resp. odd degree) of $M$. Then, for each $s\geq0$, the assignment $M\rightsquigarrow\mathscr{R}_sM$ provides an exact functor from $\mathscr{U}$ to itself.

\begin{proposition}[Chơn-Như \cite{Chon_Nhu2019}]\label{pro: R_sM subset Gamm_sM}
For $M$ an unstable $A$-module, $\mathscr{R}_sM$ is actually contained in $(\Gamma^+M)_s$.
\end{proposition}

From Lemma 2.6 and Corollary 2.7 in \cite{Chon_Nhu2019}, we obtain the following result.
\begin{proposition}\label{pro:expan R_sM in Gamma M}
Given an unstable $A$-module $M$, for any $\gamma\in \mathscr{R}_sM$, the element $\gamma$ can be expressed as follows
\[
\gamma=\sum_{I=(\epsilon_1,i_1,\dots,\epsilon_s,i_s)\in\mathcal{I}}\omega_I(-1)^{s[\frac{|m|}{2}]}u_1^{\epsilon_1}v_1^{(p-1)i_1-\epsilon_1}\cdots u_s^{\epsilon_s}v_s^{(p-1)i_s-\epsilon_s}S_s(m)
\]
where $m\in M$, $\epsilon_k=0,1$, $i_k\geq \epsilon_k$ for $1\leq k\leq s$ and $e(I)\geq |m|$.
\end{proposition}

A chain-level representation of $(\varphi_s^M)^\#$ in the Singer-Hưng-Sum chain complex is given by the following theorem.
\begin{theorem}[{Chơn-Như \cite[Theorem 3.1]{Chon_Nhu2019}}]\label{thm:representation of varphi_s in Gamma}
For any unstable $A$-module $M$, the inclusion $(\widetilde{\varphi}_s^M)^\#:\mathscr{R}_sM\to(\Gamma^+M)_s$ given by 
\[
\gamma\mapsto (-1)^{\frac{(s-2)(s-1)}{2}}\gamma
\]
 is a chain-level representation of the dual of the mod $p$ Lannes-Zarati homomorphism $(\varphi_s^M)^\#$.
\end{theorem}

In order to construct a chain-level representation of $\varphi_s^M$, we need to investigate more carefully the structure of the dual of the Singer construction.

\begin{proposition}\label{pro:additive basis of R_sM}
Given $M$ an unstable $A$-module, $\mathscr{R}_sM$ has an $\fp$-basis given by
\[
\mathscr{C}:=\left\{R_{s;0}^{\sigma_1}q_{s,0}^{j_1}\cdots R_{s;s-1}^{\sigma_s}q_{s,s-1}^{j_s}S_s(m)\right\}
\]
for all $m\in M$, $\sigma_k\in\{0,1\},j_1\in\mathbb{Z}$, $j_k\geq0,2\leq k\leq s$ and $2j_1+\sigma_1+\cdots+\sigma_s\geq|m|$.
\end{proposition}
\begin{proof}
First, we show that $\mathscr{C}$ is contained in $\mathscr{R}_sM$. 

Let $q=R_{s;0}^{\sigma_1}q_{s,0}^{j_1}\cdots R_{s;s-1}^{\sigma_s}q_{s,s-1}^{j_s}S_s(m)$ be any element satisfying the condition of the proposition, and let $\sigma_{i_1+1},\dots,\sigma_{i_k+1}$ be the set of all non-zero exponents of $R_{s;*}$s in $q$ for $0\leq i_1<\cdots<i_k\leq s-1$. 

For $k=0$ and $m\in M^{2n+\delta}$,
\[
q=\left\{
\begin{array}{ll}
\pm q_{s,0}^{j_1-n}\cdots q_{s,s-1}^{j_s}St_s(m)&\text{ if }\delta=0,\\
\pm q_{s,0}^{j_1-n-1}\cdots q_{s,s-1}^{j_s}L_s^{\frac{p-1}{2}}St_s(m)&\text{ if }\delta=1.
\end{array}
\right.
\]
Since $2j_1\geq 2n+\delta$, it implies that $q\in \mathscr{R}_sM$.

For $k>0$, using the statement (3) of Theorem \ref{thm:invariant of G}, for $m\in M^{2n+\delta}$,
\begin{align*}
q&=\pm R_{s;i_1,\cdots,i_k}q_{s,0}^{j_1+k-1}\cdots q_{s,s-1}^{j_s}S_s(m)\\
&=\pm R_{s;i_1,\cdots,i_k}q_{s,0}^{j_1+k-1-n}\cdots q_{s,s-1}^{j_s}\frac{1}{L_s^{\frac{p-1}{2}\delta}}St_s(m).
\end{align*}

If $\delta=0$, then $q=\pm q_{s,0}^{j_1+k-1-[\frac{k-1}{2}]-n}\cdots q_{s,s-1}^{j_s}R_{s;i_1,\cdots,i_k}q_{s,0}^{[\frac{k-1}{2}]} St_s(m)$. Since $2j_1+k\geq 2n$, then $j_1+k-1-[\frac{k-1}{2}]-n\geq0$. It follows that $q\in \mathscr{R}_sM$.

If $\delta=1$, then $q=\pm q_{s,0}^{j_1+k-1-[\frac{k}{2}]-n}\cdots q_{s,s-1}^{j_s}M_{s;i_1,\cdots,i_k}L_s^{\frac{p-3}{2}}q_{s,0}^{[\frac{k}{2}]} St_s(m)$. Since $2j_1+k\geq 2n+1$, then $j_1+k-1-[\frac{k}{2}]-n\geq0$. It implies that $q\in\mathscr{R}_sM$.

From the definition, for any $x\in \mathscr{R}_sM$, it can be written by $x=\lambda St_s(m)$ for some $m$, where $\lambda\in\mathscr{B}[s]$ if $|m|=2n$ and $\lambda\in\mathcal{B}[s]$ if $|m|=2n+1$. 

By the definition of $St_s(m)$, then if $|m|=2n$, then $x$ can be written by $x=(-1)^{sn}\lambda q_{s,0}^{n}S_s(m)$, where $\lambda\in\mathscr{B}[s]$.

On the other hand, by result of Chơn \cite[Proposition 3.4]{Chon2016}, $\mathscr{B}[s]$ has an $\fp$-basis consisting of all elements
\[
R_{s;0}^{\sigma_1}q_{s,0}^{j_1}\cdots R_{s;s-1}^{\sigma_s}q_{s,s-1}^{j_s},
\]
for $\sigma_k\in\{0,1\}, j_1\in\mathbb{Z},j_k\geq0, 2\leq k\leq s$ and $2j_1+\sigma_1+\cdots+\sigma_s\geq0$. 
Therefore, $x$ can be written as a linear combination of elements of $\mathscr{C}$.

Otherwise, if $|m|=2n+1$, then $x$ can be written by $x=(-1)^{sn}\lambda L_s^{\frac{p-1}{2}(2n+1)}S_s(m)$, where $\lambda$ is a sum of $f_iQ_i$ with $f_i\in D[s]$ and $Q_i=L_s^{\frac{p-1}{2}}$ or $Q_i=M_{s;i_1,\dots,i_k}L_s^{\frac{p-3}{2}+(p-1)[\frac{k}{2}]}$.

If $Q_i=L_s^{\frac{p-1}{2}}$, then $f_iQ_iL_s^{\frac{p-1}{2}(2n+1)}=f_iq_{s,0}^{n+1}$. Therefore, $f_iQ_iS_s(m)$ can be expressed as a linear combination of the needed form.

If $Q_i=M_{s;i_1,\dots,i_k}L_s^{\frac{p-3}{2}+(p-1)[\frac{k}{2}]}$, then
\[
f_iQ_iL_s^{\frac{p-1}{2}(2n+1)}=f_iR_{s;i_1,\dots,i_k}q_{s;0}^{n+[\frac{k}{2}]}.
\]

Hence, by Proposition 3.7 in \cite{Chon2016}, $f_iQ_iS_s(m)$ can be also expressed as a linear combination of the elements in $\mathscr{C}$.

Thus, $\mathscr{C}$ is a set of generators of $\mathscr{R}_sM$ as an $\fp$-vector space.

Also by the Chơn's result \cite{Chon2016}, it is easy to verify that the set $\mathscr{C}$ is linear independent.

The proof is complete.
\end{proof}

Hence, give an unstable $A$-module $M$, the Singer functor $\mathscr{R}_sM$ is an $A$-submodule of $\mathscr{B}[s]\cdot S_s(M)$. Therefore, in dual, $(\mathscr{R}_sM)^\#$ is isomorphic to a quotient of $R_s\otimes M^\#$. In order to define the structure of $(\mathscr{R}_sM)^\#$, we need the following results.

Fix a non-negative integer $s$, for any non-negative integer $n$, let $\mathscr{I}_n$ be the set of all admissible string $I=(\epsilon_1,i_1,\dots,\epsilon_s,i_s)$ satisfying $e(I)\geq n$; and let $\mathscr{J}_n$ be the set of all string $J=(\sigma_1,j_1,\dots,\sigma_s,j_s)$ satisfying the condition $\sigma_k\in\{0,1\}, 1\leq k\leq s, j_1\in \mathbb{Z}, j_k\geq0, 2\leq k\leq s$ and $2j_1+\sigma_1+\cdots+\sigma_s\geq n$.

Obviously, we obtain that
\begin{lemma}\label{lm:phi_n}
The map $\phi_n:\mathscr{J}_n\to\mathscr{I}_n$ given by $\phi_n(J)=I$ where $\epsilon_k=\sigma_k$ and
\begin{equation*}
i_k=p^{s-k}(j_1+\sigma_1+\cdots+j_k+\sigma_k)+\sum_{t=0}^{s-k-1}(p^{s-k}-p^t)(j_{k+t+1}+\sigma_{k+t+1}),
\end{equation*}
for $1\leq k\leq s$, is a bijection.
\end{lemma}
Basing on the result of Hưng-Sum \cite{Hung.Sum1995}, we get that
\begin{lemma}\label{lm:express Mui in uv} Given an unstable $A$-module $M$, for $m\in M$, let $(\sigma_1,j_1,\dots,\sigma_s,j_s)\in\mathscr{J}_{|m|}$ and $(\epsilon_1,i_1,\dots,\epsilon_s,i_s)=\phi_{|m|}(\sigma_1,j_1,\dots,\sigma_s,j_s)$. Then,
\begin{multline*}
R_{s;0}^{\sigma_1}q_{s,0}^{j_1}\cdots R_{s;s-1}^{\sigma_s}q_{s,s-1}^{j_s}S_s(m)=u_1^{\epsilon_1}v_1^{(p-1)i_s-\epsilon_1}\cdots u_s^{\epsilon_s}v_s^{(p-1)i_s-\epsilon_s}S_s(m)\\+\text{smaller monomials in the lexicographical order}.
\end{multline*}
\end{lemma}

We identify $\mathscr{R}_sM$ with its image in $R_s^\#\otimes M$ via the map $\nu_s^M$. By the identification, the dual $(\mathscr{R}_sM)^\#$ can be considered as a quotient right $A$-module of $R_s\otimes M^\#$. 

\begin{proposition}\label{pro:dual of Singer functor}
Given an unstable $A$-module $M$, the set 
\[
\mathscr{S}=\left\{Q^I\otimes \ell=\beta^{\epsilon_s}Q^{i_1}\cdots\beta^{\epsilon_s}Q^{i_s}\otimes \ell:\ell\in M^\#,I\in \mathscr{I}_{|\ell|}\right\}
\]
 represents an $\fp$-basis of $(\mathscr{R}_sM)^\#$.
\end{proposition}
\begin{proof}
By Lemma \ref{lm:express Mui in uv}, for each $\ell\in M^\#$, the set of all elements $$Q^I\otimes \ell=\beta^{\epsilon_s}Q^{i_1}\cdots\beta^{\epsilon_s}Q^{i_s}\otimes \ell$$ for $I\in \mathscr{I}_{|\ell|}$ represents a linear independent set of $(\mathscr{R}_sM)^\#$. Therefore, the set of all elements $Q^I\otimes \ell$ for $\ell\in M^\#$ and $I\in \mathscr{I}_{|\ell|}$ represents a linear independent set.

From Lemma \ref{lm:phi_n}, the number of elements of this set is equal to the dimension of $\mathscr{R}_sM$.

The proof is complete.
\end{proof}

The right $A$-module structure of $(\mathscr{R}_sM)^\#$ is induced from the Cartan formula, the right $A$-module structure of $M^\#$ and the Nishida relations.

Observation from Proposition \ref{pro:expan R_sM in Gamma M} that the elements $Q^I\otimes \ell\in R_s\otimes M^\#$ of $e(I)<|\ell|$ represents a trivial element in $(\mathscr{R}_sM)^\#$. 

Taking dual Theorem \ref{thm:representation of varphi_s in Gamma}, we have the following result.
\begin{proposition}\label{pro:representation of vaphi_s in lambda}
For any unstable $A$-module $M$, the projection $$\widetilde{\varphi}_s^M:\Lambda_s\otimes M^\#\to (\mathscr{R}_sM)^\#$$ given by
\[
\lambda_I\otimes \ell\to (-1)^{\frac{(s-1)(s-2}{2}}[Q^I\otimes \ell]
\]
is a chain-level representation of the mod $p$ Lannes-Zarati homomorphism $\varphi_s^M$.
\end{proposition}
\section{The power operations}
 This section is devoted to develop the power operations, these are useful tools to study the behavior of the Lannes-Zarati homomorphism in the next section.

From Liulevicius \cite{Liulevicius1960}, \cite{Liu62} and May \cite{May70}, there exists a power operation $\p^0:\Ext{A}{s,s+t}{\fp}{\fp}\to \Ext{A}{s,p(s+t)}{\fp}{\fp}$. Its chain-level representation in $\Lambda$ is given by
\[
\widetilde{\p}^0(\lambda^{\epsilon_1}_{i_1-1}\cdots\lambda^{\epsilon_s}_{i_s-1})=\left\{
\begin{array}{ll}
\lambda^{\epsilon_1}_{pi_1-1}\cdots\lambda^{\epsilon_s}_{pi_s-1},& \epsilon_1=\cdots=\epsilon_s=1,\\
0,& \text{otherwise}.
\end{array}
\right.
\]
\begin{lemma}\label{lm:power operation on R}
The operation $\widetilde{\p}^0$ induces an operation, which is also denoted by $\widetilde{\p}^0$, on the Dyer-Lashof algebra $R$ given by
\[
\widetilde{\p}^0(\beta^{\epsilon_1}Q^{i_1}\cdots\beta^{\epsilon_s}Q^{i_s})=\left\{
\begin{array}{ll}
\beta^{\epsilon_1}Q^{pi_1}\cdots\beta^{\epsilon_s}Q^{pi_s},& \epsilon_1=\cdots=\epsilon_s=1,\\
0,& \text{otherwise}.
\end{array}
\right.
\]
\end{lemma}
\begin{proof}
It is sufficient to show that if $\lambda^1_{i_1-1}\cdots\lambda^1_{i_s-1}$ has negative excess then so does $\lambda^1_{pi_1-1}\cdots\lambda^1_{pi_s-1}$ for $s\geq2$.

By inspection, one gets
\begin{align*}
e(\lambda^1_{pi_1-1}\cdots\lambda^1_{pi_s-1})&=2pi_1-\sum_{k=2}^s2p(p-1)i_k+(s-2)\\
&=pe(\lambda^1_{i_1-1}\cdots\lambda^1_{i_s-1})-(p-1)(s-2).
\end{align*}

Therefore, if $e(\lambda^1_{i_1-1}\cdots\lambda^1_{i_s-1})<0$ then $e(\lambda^1_{pi_1-1}\cdots\lambda^1_{pi_s-1})<0$.
\end{proof}
\begin{lemma}\label{lm:power operation on Ann(R)}
The operation $\widetilde{\p}^0$ commutes with the action of $A$. In particular,
\begin{equation}\label{eq:power operation on Ann(R)}
\widetilde{\p}^0((\beta^{\epsilon_1}Q^{i_1}\cdots\beta^{\epsilon_s}Q^{i_s})P^k)=(\widetilde{\p}^0(\beta^{\epsilon_1}Q^{i_1}\cdots\beta^{\epsilon_s}Q^{i_s}))P^{pk}.
\end{equation}
\end{lemma}
\begin{proof}
It is sufficient to show the assertion of lemma in the case $\epsilon_1=\cdots=\epsilon_s=1$.

We will prove by induction on $s$.

For $s=1$, it is easy to see that
\begin{align*}
\widetilde{\p}^0((\beta Q^i)P^k)&=\widetilde{\p}^0((-1)^k\binom{(p-1)(i-k)-1}{k}\beta Q^{i-k})\\
&=(-1)^k\binom{(p-1)(i-k)-1}{k}\beta Q^{pi-pk},
\end{align*}
and
\begin{align*}
(\widetilde{\p}^0(\beta Q^i))P^{pk}=\beta Q^{pi}P^{pk}
=(-1)^{pk}\binom{(p-1)(pi-pk)-1}{pk}\beta Q^{pi-pk}.
\end{align*}
Since $(-1)^{pk}\binom{(p-1)(pi-pk)-1}{pk}\equiv(-1)^k\binom{(p-1)(i-k)-1}{k}\mod p$, we have the assertion.

For $s>1$, by the inductive hypothesis,
\begin{align*}
\widetilde{\p}^0&((\beta Q^{i_1}\cdots\beta Q^{i_s})P^k)\\
&=\widetilde{\p}^0\left(\sum_t(-1)^{k+t}\binom{(p-1)(i_1-k)-1}{k-pt}\beta Q^{i_1-k+t}(\beta Q^{i_2}\cdots\beta Q^{i_s})P^t\right)\\
&\quad +\widetilde{\p}^0\left(\sum_t(-1)^{k+t}\binom{(p-1)(i_1-k)-1}{k-pt-1}Q^{i_1-k+t}(\beta Q^{i_2}\cdots\beta Q^{i_s})\beta P^t\right)\\
&=\sum_t(-1)^{k+t}\binom{(p-1)(i_1-k)-1}{k-pt}\beta Q^{p(i_1-k+t)}(\beta Q^{pi_2}\cdots\beta Q^{pi_s})P^{pt}.
\end{align*}
On the other hand,
\begin{align*}
(\widetilde{\p}^0&(\beta Q^{i_1}\cdots\beta Q^{i_s}))P^{pk}=(\beta Q^{pi_1}\cdots\beta Q^{pi_s})P^{pk}\\
&=\sum_j(-1)^{pk+j}\binom{(p-1)(pi_1-pk)-1}{pk-pj}\beta Q^{pi_1-pk+j} (\beta Q^{pi_2}\cdots \beta Q^{pi_s})P^j\\
&\quad + \sum_j(-1)^{pk+j}\binom{(p-1)(pi_1-pk)-1}{pk-pj-1}Q^{pi_1-pk+j} (\beta Q^{pi_2}\cdots \beta Q^{pi_s})\beta P^j\\
&=\sum_j(-1)^{k+j}\binom{(p-1)(i_1-k)-1}{k-j}\beta Q^{pi_1-pk+j} (\beta Q^{pi_2}\cdots \beta Q^{pi_s})P^j.
\end{align*}

If $j$ is not divisible by $p$ then $(p-1)(pi_2-j)-1\equiv j-1\mod p$, while $j-p\ell\equiv j\mod p$. Therefore,
\begin{align*}
(\beta &Q^{pi_2}\cdots \beta Q^{pi_s})P^j\\
&=\sum_j(-1)^{j+\ell}\binom{(p-1)(pi_2-\ell)-1}{j-p\ell}\beta Q^{pi_2-j+\ell} (\beta Q^{pi_3}\cdots \beta Q^{pi_s})P^\ell\\
&\quad + \sum_j(-1)^{j+\ell}\binom{(p-1)(pi_2-j)-1}{j-p\ell-1}Q^{pi_2-j+\ell} (\beta Q^{pi_3}\cdots \beta Q^{pi_s})\beta P^j\\
&=\sum_j(-1)^{j+\ell}\binom{(p-1)(pi_2-\ell)-1}{j-p\ell}\beta Q^{pi_2-j+\ell} (\beta Q^{pi_3}\cdots \beta Q^{pi_s})P^\ell=0.
\end{align*}

Thus,
\begin{align*}
(\widetilde{\p}^0&(\beta Q^{i_1}\cdots\beta Q^{i_s}))P^{pk}\\
&=\sum_j(-1)^{k+t}\binom{(p-1)(i_1-k)-1}{k-pt}\beta Q^{p(i_1-k+t)} (\beta Q^{pi_2}\cdots \beta Q^{pi_s})P^{pt}.
\end{align*}

The lemma is proved.
\end{proof}

It is easy to see that if $(q)\beta = 0$ then $(\widetilde{\p}^0(q)\beta)=0$ for $q\in R$. This fact together with Lemma \ref{lm:power operation on Ann(R)} show that the operation $\widetilde{\p}^0$ induces a power operation on $(\fp\tsor{A}\mathscr{R}_s\fp)^\#$, which is also denoted by $\p^0$.

Similarly, the power operation $\widetilde{\p}^0$ acting on $\Lambda\otimes H$ mentioned in Section \ref{sec:cohomology} also induces a power operation on $(\fp\tsor{A}\mathscr{R}_sP)^\#$ which is also denoted by $\p^0$. 

\begin{proposition}\label{pro:power operations}
The power operations $\p^0$s commute with each other through the Lannes-Zarati homomorphism. In other words, the following diagram is commutative
\[
\xymatrix{
\Ext{A}{s,s+t}{M}{\fp}\ar[r]^{\p^0}\ar[d]_{\varphi_s^M}&\Ext{A}{s,p(s+t)}{M}{\fp}\ar[d]_{\varphi_s^M}\\
(\fp\tsor{A}\mathscr{R}_sM)^\#_t\ar[r]^{\p^0}&(\fp\tsor{A}\mathscr{R}_sM)^\#_{p(s+t)-s},
}
\]
for $M=\fp$ and $M=P$.
\end{proposition}
\begin{proof}
It is immediate from Proposition \ref{pro:representation of vaphi_s in lambda}.
\end{proof}

\section{The behavior of the mod $p$ Lannes-Zarati homomorphism}
In this section, we use the chain-level representation map of the $\varphi_s^M$ constructed in the previous section to investigate its behavior.
\subsection{The behavior of $\varphi_s^{\fp}$}
\begin{theorem}\label{thm:rank3}
The third Lannes-Zarati homomorphism
\[
\varphi_3^{\fp}:\Ext{A}{3,3+t}{\fp}{\fp}\to (\fp\tsor{A}\mathscr{R}_3\fp)^\#_t
\]
is a monomorphism for $t=0$ and vanishing for all $t>0$.
\end{theorem}
From Proposition \ref{pro:representation of vaphi_s in lambda}, we get the following lemma.
\begin{lemma}\label{lm:of theorem rank3}
If $\lambda_I\in \Lambda_s$ and $\lambda_J\in\Lambda_{\ell}$ such that $\widetilde{\varphi}_s^{\fp}(\lambda_I)=0$ or $\widetilde{\varphi}_{\ell}^{\fp}(\lambda_J)=0$ then $\widetilde{\varphi}_{s+\ell}^{\fp}(\lambda_I\lambda_J)=0$.
\end{lemma}
\begin{proof}[Proof of Theorem \ref{thm:rank3}]
By the results of Liulevicius \cite{Liu62} and Aikawa \cite{Aikawa1980}, $\Ext{A}{3,3+t}{\fp}{\fp}$ is spanned by following elements (for convenience we will write $\ext{A}{s,s+t}$ for $\Ext{A}{s,s+t}{\fp}{\fp}$)
\begin{enumerate}
\item $h_ih_jh_k=[\lambda^1_{p^i-1}\lambda^1_{p^j-1}\lambda^1_{p^k-1}]\in\ext{A}{3,2(p-1)(p^i+p^j+p^k)},0\leq i\leq j-2\leq k-4$;
\item $\alpha_0h_ih_j=[\lambda^1_{p^i-1}\lambda^1_{p^j-1}\lambda^0_{-1}]\in\ext{A}{3,2(p-1)(p^i+p^j)+1},1\leq i\leq j-2$;
\item $\alpha_0^2h_i=[\lambda^1_{p^i-1}(\lambda^0_{-1})^2]\in\ext{A}{3,2(p-1)p^i+2}, i\geq 1$;
\item $\alpha_0^3=[(\lambda^0_{-1})^3]\in \ext{A}{3,3}$;
\item $\tilde{\lambda}_ih_j=[L_i\lambda^1_{p^j}]\in\ext{A}{3,2(p-1)(p^{i+1}+p^j)}, i,j\geq0, j\neq i+2$;
\item $\tilde{\lambda}_i\alpha_0=[L_i\lambda^0_{-1}], i\geq0$;
\item $h_{i;1,2}h_j=[\lambda^1_{p^{i+1}-1}\lambda^1_{2p^i-1}\lambda^1_{p^j-1}]\in\ext{A}{3,2(p-1)(p^{i+1}+2p^i+p^j)}, i,j\geq0, j\neq i+2,i,i-1$;
\item $h_{i;1,2}\alpha_0=[\lambda^1_{p^{i+1}-1}\lambda^1_{2p^i-1}\lambda^0_{-1}]\in\ext{A}{3,2(p-1)(p^{i+1}+2p^i)+1}, i\geq1$;
\item $h_{i;2,1}h_j=[\lambda^1_{2p^{i+1}-1}\lambda^1_{p^i-1}\lambda^1_{p^j-1}]\in\ext{A}{2(p-1)(2p^{i+1}+p^i+p^j)};i,j\geq0, j\neq i+2,i\pm1,i$;
\item $h_{i;2,1}\alpha_0=[\lambda^1_{2p^{i+1}-1}\lambda^1_{p^i-1}\lambda^0_{-1}]\in\ext{A}{3,2(p-1)(2p^{i+1}+p^i)+1},i\geq1$;
\item $h_j\rho=[\lambda^1_{p^j-1}\lambda^1_1\lambda^0_{-1}]\in\ext{A}{3,2(p-1)(p^j+2)+1},j\geq 2$;
\item $h_{i;3,2,1}=(\p^0)^i[\lambda^1_{3p^{2}-1}\lambda^1_{2p-1}\lambda^1_{0}]\in\ext{A}{3,2(p-1)(3p^{i+2}+2p^{i+1}+p^i},p\neq 3,i\geq0$;
\item $h'_{3,2,1}=[\lambda^1_{3p-1}\lambda^1_{1}\lambda^0_{-1}]\in\ext{A}{3,2(p-1)(3p+2)+1},p\neq 3$;
\item $h_{i;2,2,1}=(\p^0)^i[\lambda^1_{2p^{3}-1}\lambda^1_{2p-1}\lambda^1_{0}]\in\ext{A}{3,2(p-1)(2p^{i+3}+2p^{i+1}+p^i)},p=3,i\geq0$;
\item $h'_{2,2,1}=[\lambda^1_{2p^2-1}\lambda^1_{1}\lambda^0_{-1}]\in\ext{A}{3,2(p-1)(2p^2+2)+1},p=3$;
\item $h_{i;1,3,1}=(\p^0)^i[\lambda^1_{p^{2}-1}\lambda^1_{3p-1}\lambda^1_{0}]\in\ext{A}{3,2(p-1)(p^{i+2}+3p^{i+1}+p^i)},p\neq 3, i\geq0$;
\item $h'_{1,3,1}=[\lambda^1_{p-1}\lambda^1_{2}\lambda^0_{-1}]\in\ext{A}{3,2(p-1)(p+3)+1},p\neq 3$;
\item $h_{i;2,1,2}=(\p^0)^i[\lambda^1_{2p^{2}-1}\lambda^1_{p-1}\lambda^1_{1}]\in\ext{A}{3,2(p-1)(2p^{i+2}+p^{i+1}+2p^i},i\geq0$;
\item $h_{i;1,2,3}=(\p^0)^i[\lambda^1_{p^2-1}\lambda^1_{2p-1}\lambda^1_{2}]\in\ext{A}{3,2(p-1)(p^{i+2}+2p^{i+1}+3p^i)},p\neq 3, i\geq0$;
\item $\varrho_3=[\lambda^1_{2}(\lambda^0_{-1})^2]\in\ext{A}{3,6(p-1)+2}, p\neq 3$;
\item $\varrho'_3=[\lambda^1_5(\lambda^0_{-1})^2]\ext{A}{3,12(p-1)+2},p=3$;
\item $f_i=(\p^0)^{i-1}[M]\in\ext{A}{3,2(p-1)(p^{i+1}+2p^i)},i\geq 1$;
\item $g_i=(\p^0)^{i-1}[N]\in\ext{A}{3,2(p-1)(2p^{i+1}+p^i)},i\geq 1$;
\end{enumerate}
where
\begin{align*}
L_i&=(\p^0)^i\left(\sum_{j=1}^{(p-1)}\frac{(-1)^{j+1}}{j}\lambda^1_{(p-j)-1}\lambda^1_{j-1}\right), i\geq0;\\
M&=\sum_{j=1}^{p-1}\frac{(-1)^{j+1}}{j}(\lambda^1_{jp-1}\lambda^1_{(p^2-jp)-1}\lambda^1_{2p-1}-2\lambda^1_{p^2-1}\lambda^1_{j-1}\lambda^1_{2p-j-1}\\
&\quad\quad\quad-2\lambda^1_{p^2-1}\lambda^1_{p+j-1}\lambda^1_{p-j-1});\\
N&=\sum_{j=1}^{p-1}\frac{(-1)^{j+1}}{j}(2\lambda^1_{jp-1}\lambda^1_{(2p^2-jp)-1}\lambda^1_{p-1}+2\lambda^1_{p^2+jp-1}\lambda^1_{p^2-jp-1}\lambda^1_{p-1}\\
&\quad\quad\quad-\lambda^1_{2p^2-1}\lambda^1_{j-1}\lambda^1_{p-j-1}).
\end{align*}

Observe that the elements $h_ih_jh_k\ (0\leq i\leq j-2\leq k-4)$, $\alpha_0h_ih_j\ (1\leq i\leq j-2)$, $h_{i;1,2}h_j$, $h_{i;1,2}\alpha_0$, $h_{0;1,3,1}$, $h'_{1,3,1}$, $h_{0;1,2,3},$ and $f_0$ are represented by cycles of negative excess. Therefore, their images under $\widetilde{\varphi}_3^{\fp}$ are trivial, so that their images under $\varphi_3^{\fp}$ are also trivial.

By Proposition \ref{pro:power operations},
\[
\varphi_3^{\fp}(h_{i;1,3,1})=\varphi_3^{\fp}((\p^0)^i(h_{0;1,3,1}))=(\p^0)^i(\varphi_3^{\fp}(h_{0;1,3,1}))=0, i\geq0.
\]

By the same argument, we get $\varphi_3^{\fp}(h_{i;1,2,3})=0$ for $i\geq0$ and $\varphi_3^{\fp}(f_i)=0$ for $i\geq 1$.

From the proof of \cite[Theorem 4.2]{Chon_Nhu2019}, we have that $\widetilde{\varphi}_2^{\fp}(\lambda^1_{2p^{i+1}-1}\lambda^1_{p^i-1})=0$ for $i\geq0$ and $\widetilde{\varphi}_2^{\fp}(\lambda^1_1\lambda^0_{-1})=0$. Therefore, basing on Lemma \ref{lm:of theorem rank3}, we obtain that 
\begin{itemize}
\item $\widetilde{\varphi}_3^{\fp}(\lambda^1_{2p^{i+1}-1}\lambda^1_{p^i-1}\lambda^1_{p^j-1})=0$, 
\item $\widetilde{\varphi}_3^{\fp}(\lambda^1_{2p^{i+1}-1}\lambda^1_{p^i-1}\lambda^0_{-1})=0$,
\item $\widetilde{\varphi}_3^{\fp}(\lambda^1_{p^j-1}\lambda^1_1\lambda^0_{-1})=0$, 
\item $\widetilde{\varphi}_3^{\fp}(\lambda^1_{3p-1}\lambda^1_{1}\lambda^0_{-1})=0$, and
\item $\widetilde{\varphi}_3^{\fp}(\lambda^1_{2p^2-1}\lambda^1_{1}\lambda^0_{-1})=0$.
\end{itemize}

It follows that the images of elements $h_{i;2,1}h_j$, $h_{i;2,1}\alpha_0$, $h_j\rho$, $h'_{3,2,1}$ and $h'_{2,2,1}$ are trivial under the map $\varphi_3^{\fp}$.

Basing on \cite[Theorem 4.2]{Chon_Nhu2019}, one gets
$\widetilde{\varphi}_2^{\fp}(L_i)=\beta Q^{(p-1)p^i}\beta Q^{p^i}.$ Therefore, we obtain
\begin{itemize}
\item $\widetilde{\varphi}_3^{\fp}(L_i\lambda_{p^j})=-\beta Q^{(p-1)p^i}\beta Q^{p^i}\beta Q^{p^j}$, and
\item $\widetilde{\varphi}_3^{\fp}(L_i\lambda^0_{-1})=-Q^{(p-1)p^i}\beta Q^{p^i} Q^0$.
\end{itemize}

Since the right hand side of the first formula is of negative excess for $i,j\geq0$, it implies that $\varphi_3^{\fp}(\tilde{\lambda}_ih_j)=0$.

By the proof of \cite[Theorem 4.2]{Chon_Nhu2019}, $\beta Q^i Q^0=0\in R_2$, so that the right hand side of the second formula is trivial in $R_3$. It follows $\varphi_3^{\fp}(\alpha_0\tilde{\lambda}_i)=0$. By the same method, we also get $\varphi_3^{\fp}(\alpha_0^2h_i)=0$.

From above, $\widetilde{\varphi}_2^{\fp}(\lambda^1_{2p-1}\lambda^1_{0})=0$ and $\widetilde{\varphi}_2^{\fp}(\lambda^1_{p-1}\lambda^1_1)=0$, therefore,
\begin{itemize}
\item $\widetilde{\varphi}_3^{\fp}(\lambda^1_{3p^{2}-1}\lambda^1_{2p-1}\lambda^1_{0})=0$,
\item $\widetilde{\varphi}_3^{\fp}(\lambda^1_{2p^{3}-1}\lambda^1_{2p-1}\lambda^1_{0})=0$, and
\item $\widetilde{\varphi}_3^{\fp}(\lambda^1_{2p^{2}-1}\lambda^1_{p-1}\lambda^1_{1})=0$.
\end{itemize}
It follows that, using Proposition \ref{pro:power operations}, $\varphi_3^{\fp}(h_{0;3,2,1})=0$, $\varphi_3^{\fp}(h_{0;2,2,1})=0$ and $\varphi_3^{\fp}(h_{0;2,1,2})=0$, therefore, $\varphi_3^{\fp}(h_{i;3,2,1})=0$, $\varphi_3^{\fp}(h_{i;2,2,1})=0$ and $\varphi_3^{\fp}(h_{i;2,1,2})=0$ for $i\geq0$.

Applying the adem relation, in $R_2$, one gets $\beta Q^3 Q^0=0$ and $\beta Q^6Q^0=0$, therefore,
\begin{itemize}
\item $\widetilde{\varphi}_3^{\fp}(\lambda^1_{2}(\lambda^0_{-1})^2)=-\beta Q^3 Q^0 Q^0=0$, and
\item $\widetilde{\varphi}_3^{\fp}(\lambda^1_5(\lambda^0_{-1})^2)=-\beta Q^6Q^0Q^0=0$.
\end{itemize}

Hence, $\varphi_3^{\fp}(\varrho_3)=0$ and $\varphi_3^{\fp}(\varrho'_3)=0$.

By inspection, using Proposition \ref{pro:representation of vaphi_s in lambda}, one gets
\begin{multline*}
\widetilde{\varphi}_3(N)=\\-\sum_{j=1}^{p-1}\frac{(-1)^{j}}{j}\left(2\beta Q^{jp}\beta Q^{2p^2-jp}\beta Q^{p}+2\beta Q^{p^2+jp}\beta Q^{p^2-jp}\beta Q^p-\beta Q^{2p^2}\beta Q^{j}\beta Q^{p-j}\right).
\end{multline*}

Since two first terms of the right hand side of the formula are of negative excess, then
\[
\widetilde{\varphi}_3(N)=\sum_{j=1}^{p-1}\frac{(-1)^{j}}{j}\beta Q^{2p^2}\beta Q^{j}\beta Q^{p-j}.
\]

It is easy to verify that $\beta Q^j\beta Q^{p-j}=0$ for $j<p-1$, therefore,
\[
\widetilde{\varphi}_3(N)=-\beta Q^{2p^2}\beta Q^{p-1}\beta Q^{1}.
\]

Applying the adem relation, we get
\[
\beta Q^{2p^2}\beta Q^{p-1}=\sum_{i}(-1)^{2p^2+i+1}\binom{(p-1)(i+1-p)-1}{pi-2p^2-1}\beta Q^{2p+p-1-i}\beta Q^i.
\]

Since $pi\geq 2p^2+1$, then $pi=2p^2+pa$ for some $a\geq1$. In this case, we get 
\begin{align*}
e(\beta Q^{2p+p-1-i}\beta Q^i)&=2(2p^2+p-1-i)-2(p-1)i=4p^2+2p-2-2pi\\
&=4p^2+2p-2-4p^2-2pa=2p-2pa-2<0.
\end{align*}
Therefore, $\beta Q^{2p^2}\beta Q^{p-1}=0$, and then $\widetilde{\varphi}_3(N)=0$. It implies that $\varphi_3^{\fp}(g_1)=0$ and, hence, $\varphi_3^{\fp}(g_i)=0$ for $i\geq1$.

Finally, it is easy to check that $\varphi_3(\alpha_0^3)=-Q^0Q^0Q^0\neq 0\in R_3$.

The proof is complete.
\end{proof}

\subsection{The behavior of $\varphi_s^{P}$}
\begin{theorem}\label{thm:varphi_0^P}
The Lannes-Zarati homomorphism $\varphi_0^{P}:\Ext{A}{0,t}{P}{\fp}\to (\fp\tsor{A}\mathscr{R}_0P)^\#_t$ is an isomorphism. 
\end{theorem}
\begin{proof}
It is easy to see that $(\fp\tsor{A}\mathscr{R}_0P)^\#$ is spanned by $ab^{[kp^i-1]}$ for $i\geq0$ and $1\leq k\leq p-1$. Therefore, the assertion of the theorem is followed from Theorem \ref{thm:Ext(P)^0} and Proposition \ref{pro:representation of vaphi_s in lambda}.
\end{proof}
\begin{theorem}\label{thm:varphi_1^P}
The Lannes-Zarati homomorphism $\varphi_1^{P}:\Ext{A}{1,1+t}{P}{\fp}\to (\fp\tsor{A}\mathscr{R}_1P)^\#_t$ sends
\begin{enumerate}
\item $h_i\widehat{h}_i(1)$ to $\left[\beta Q^{p^i}ab^{[p^i-1]}\right]$, for $i\geq0$;
\item $h_i\widehat{h}_j$ to $\left[\beta Q^{p^i}ab^{[(p-1)p^j-1]}\right]$ for $0\leq j<i$;
\item $h_i\widehat{h}_j(k)$ to $\left[\beta Q^{p^i}ab^{[kp^j-1]}\right]$ for $0\leq j<i$, $1\leq k<p-1$;
\item $\widehat{k}_i(k)$ to $(\p^0)^i\left(
\left[\beta Q^{k+1}ab^{[k]}\right]\right),i\geq0,1\leq k<p-1$; and
\item others to zero.
\end{enumerate}
\end{theorem}
\begin{proof}
By Theorem \ref{thm: Ext(P)^1}, $\Ext{A}{1,1+t}{P}{\fp}$ is spanned by
\begin{enumerate}
\item $\alpha_0\widehat{h}_i=\left[\lambda^0_{-1}ab^{[(p-1)p^i-1]}\right],i\geq1$;
\item $\alpha_0\widehat{h}_{i}(k)=\left[\lambda^0_{-1}ab^{[kp^i-1]}\right],i\geq1,1\leq k<p-1$;
\item $\widehat{\alpha}(\ell)=\left[\lambda^0_{-1}ab^{[p+\ell]}+(\ell+1)\lambda^0_0ab^{\ell+1}\right],0\leq\ell<p-2$;
\item $h_i\widehat{h}_i(1)=\left[\lambda^1_{p^i-1}ab^{[p^i-1]}\right], i\geq0$;
\item $h_i\widehat{h}_j=\left[\lambda^1_{p^i-1}ab^{[(p-1)p^j-1]}\right],i,j\geq0, j\neq i, i+1$;
\item $h_i\widehat{h}_{j}(k)=\left[\lambda^1_{p^i-1}ab^{[kp^j-1]}\right], i,j\geq0, j\neq i, i+1,1\leq r<p-1$;
\item $\widehat{d}_i(k)=(\p^0)^{i-1}\left(\left[\lambda^1_{p-1}ab^{[kp+p-2]}\right]\right),i\geq1,1\leq k\leq p-1$;
\item $\widehat{k}_i(k)=(\p^0)^i\left(
\left[\sum_{j=0}^{k}\frac{1}{j+1}\lambda^{1}_{j}ab^{[(k-j)p+j]}\right]\right),i\geq0,1\leq k<p-1$;
\item $\widehat{p}_i(r)=(\p^0)^i\left(\left[\sum_{j=0}^{p-1-r}\frac{\binom{r+j}{j}}{j+1}\lambda^1_{j}ab^{[(p-j-1)p+r+j]}\right]\right),i\geq0$, $1\leq r<p-1$.
\end{enumerate}

Using Proposition \ref{pro:representation of vaphi_s in lambda}, it is easy to verify that the images of the following elements
\begin{itemize}
\item $\alpha_0 \widehat{h}_i$ for $i\geq1$;
\item $\alpha_0\widehat{h}_i(k)$ for $i\geq1, 1\leq k<p-1$;
\item $\widehat{\alpha}(\ell)$ for $0\leq \ell<p-2$;
\item $h_i\widehat{h}_j$ for $0\leq i< j-1$; and
\item $h_i\widehat{h}_j(k)$ for $0\leq i<j-1, 1\leq k<p-1$ 
\end{itemize}
 are trivial.
 
 By inspection, using Proposition \ref{pro:representation of vaphi_s in lambda}, we get
 \[
 \widetilde{\varphi}_1^P(\lambda^1_{p-1}ab^{[kp+p-2]})=\left[\beta Q^{p}ab^{[kp+p-2]}\right].
 \]
 
Since $2p-(2(kp+p-2)+1)<0$ for all $k\geq 1$, it follows that $\varphi_1^P(\widehat{d}_1(k))=0$ for all $1\leq k\leq p-1$. Using Proposition \ref{pro:power operations}, we obtain
\[
\varphi_1^P(\widehat{d}_i(k))=(\p^0)^{i-1}\varphi_1^P(\widehat{d}_1(k))=0.
\]

By the same argument, since $2(j+1)-(2((p-1-j)p+r+j)+1)<0$ for all $0\leq j\leq p-1-r$ and $r\geq 1$, it implies $\varphi_1^P(\widehat{p}_0(r))=0$. In addition, using Proposition \ref{pro:power operations}, we get $\varphi_1^P(\widehat{p}_i(r))=0$.

Finally, using Proposition \ref{pro:representation of vaphi_s in lambda}, it is easy verify that, in $(\mathscr{R}_1P)^\#$,
\begin{itemize}
\item $\varphi_1^P(h_i\widehat{h}_i(1))=\left[\beta Q^{p^i}ab^{[p^i-1]}\right]\neq0$, for $i\geq0$;
\item $\varphi_1^P(h_i\widehat{h}_j)=\left[\beta Q^{p^i}ab^{[(p-1)p^j-1]}\right]\neq0$ for $0\leq j<i$;
\item $\varphi_1^P(h_i\widehat{h}_j(k))=\left[\beta Q^{p^i}ab^{[kp^j-1]}\right]\neq0$ for $0\leq j<i$, $1\leq k<p-1$; and
\item $\varphi_1^P(\widehat{k}_i(k))=(\p^0)^i\left(
\left[\beta Q^{k+1}ab^{[k]}\right]\right)\neq 0,i\geq0,1\leq k<p-1$.
\end{itemize}

The proof is complete.
\end{proof}
\begin{remark}\label{rm:varphi_1P}\rm
It is easy to see that $[\beta Q^{p-1}b^{[1]}+Q^{p-1}a]$ is non-trivial in $(\fp\tsor{A}\mathscr{R}_1P)^\#$. It follows that, basing on Theorem \ref{thm:varphi_1^P}, $\varphi_1^P$ is not an epimorphism. This fact is similar to the case $p=2$ (see Remark \ref{rm:vaphi_1^P for p=2}).
\end{remark}
\subsection*{Acknowledgement}The authors would like to thank Lê Minh Hà and Jean Lannes for many
fruitful discussions. 
The paper was completed while the first author was visiting the Vietnam Institute for Advanced
Study in Mathematics (VIASM). He thanks the VIASM for support and hospitality.
\appendix
\section{The Singer transfer}\label{appendix}%
The algebraic Singer transfer is first constructed by Singer for $p=2$ \cite{Singer1989}, and it is later generalized for $p$ odd by Crossley \cite{Crossley1999}. It plays an important role in study of the cohomology of the Steenrod algebra (see \cite{Singer1989}, \cite{Bruner.etal2005}, \cite{Hung05}, \cite{Quy07}, \cite{Ha2007}, \cite{Nam2008}, \cite{Chon.Ha2010}, \cite{Chon.Ha2011}, \cite{Chon.Ha2012}, \cite{Chon.Ha2014}). Let us recall the construction of the algebraic Singer transfer.

For any $A$-module $M$, let $$\delta_1(\Sigma^{-(s-i)}P_{i}\otimes M):\Tor{A}{r}{\fp}{\Sigma^{-(s-i)}P_{i}\otimes M}\to\Tor{A}{r-1}{\fp}{\Sigma^{-(s-i-1)}P_{i+1}\otimes M}$$ be the connecting homomorphism associated to the short exact sequence
\[
0\to \Sigma^{-(s-i-1)}P_{i+1}\otimes M\to \Sigma^{-(s-i-1)}P_{i}\otimes\hat{P}\otimes M\to \Sigma^{-(s-i)}P_{i}\otimes M\to 0,
\]
for $0\leq i\leq s-1$ with the convention $P_0=\fp$.

The dual of the Singer transfer
\[
(\psi_s^M)^\#:\Tor{A}{s,s+t}{\fp}{M}\cong\Tor{A}{s,t}{\fp}{\Sigma^{-s}M}\to\Tor{A}{0,t}{\fp}{P_s\otimes M}
\]
is define by
\[
(\psi_s^M)^\#=\delta_1(\Sigma^{-1}P_{s-1}\otimes M)\circ\cdots\circ\delta_1(\Sigma^{-s}M).
\]

Taking dual, we have a homomorphism (called the algebraic Singer transfer), for each $s\geq0$,
\[
\psi_s^M:\Ext{A}{0,t}{P_s\otimes M}{\fp}\to\Ext{A}{s,s+t}{M}{\fp}.
\]
\begin{proposition}\label{pro:transfer in Gamma}
The map $(\widetilde{\psi}_s^M)^\#:(\Gamma^+M)_s\to P_s\otimes M$ induced by
\begin{multline*}
u_1^{\epsilon_1}v_1^{(p-1)i_1-\epsilon_1}\cdots u_s^{\epsilon_s}v_s^{(p-1)i_s-\epsilon_s}S_s(m)\mapsto\\
(-1)^{\frac{s(s+1)}{2}+i_1+\cdots+i_s}
\beta^{1-\epsilon_1}P^{i_1}(x_1y_1\otimes (\cdots(\beta^{1-\epsilon_s}P^{i_s}(x_sy_s^{-1}\otimes m))\cdots))
\end{multline*}
is a chain-level representation of $(\psi_s^M)^\#$.
\end{proposition}
\begin{proof}
Let $\gamma=\sum_{I=(i_1,\epsilon_1,\dots,i_s,\epsilon_s)\in \mathcal{I}}\omega_I u_1^{\epsilon_1}v_1^{(p-1)i_1-\epsilon_1}\cdots u_s^{\epsilon_s}v_s^{(p-1)i_s-\epsilon_s}S_s(\Sigma^{-s}m)$ be a cycle in $(\Gamma^+\Sigma^{-s}M)_s$. Then, $\gamma$ can be pulled back by
\[
\gamma_1=\sum_{I\in \mathcal{I}}\omega_Iu_1^{\epsilon_1}v_1^{(p-1)i_1-\epsilon_1}\cdots u_s^{\epsilon_s}v_s^{(p-1)i_s-\epsilon_s}S_s(\Sigma^{-(s-1)}x_sy_s^{-1}\otimes m)\in(\Gamma^+\Sigma^{-(s-1)}\widehat{P}\otimes M)_s.
\]
Then, in $(\Gamma^+\Sigma^{-(s-1)}\hat{P}\otimes M)_s$,
\begin{multline*}
\partial(\gamma_1)=
\sum_{I\in \mathcal{I}}\omega_I(-1)^{\epsilon_1+\cdots+\epsilon_s+i_s+s+(s-1)\epsilon_s}\\
\times u_1^{\epsilon_1}v_1^{(p-1)i_1-\epsilon_1}\cdots u_{s-1}^{\epsilon_{s-1}}v_{s-1}^{(p-1)i_{s-1}-\epsilon_{s-1}}S_{s-1}(\Sigma^{-(s-1)}\beta^{1-\epsilon_s}P^{i_s}(x_sy_s^{-1}\otimes m)).
\end{multline*}

It follows that
\begin{multline*}
\delta_1(\Sigma^{-s}M)(\gamma)=\sum_{I\in \mathcal{I}}\omega_I(-1)^{\epsilon_1+\cdots+\epsilon_s+i_s+s+(s-1)\epsilon_s}\\
\times u_1^{\epsilon_1}v_1^{(p-1)i_1-\epsilon_1}\cdots u_{s-1}^{\epsilon_{s-1}}v_{s-1}^{(p-1)i_{s-1}-\epsilon_{s-1}}S_{s-1}(\Sigma^{-(s-1)}\beta^{1-\epsilon_s}P^{i_s}(x_sy_s^{-1}\otimes m)).
\end{multline*}

By induction, we get
\begin{multline*}
(\widetilde{\psi}_s^M)^\#(\gamma)=\sum_{I\in\mathcal{I}}\omega_I(-1)^{s(\deg\gamma+\deg m)+\frac{s(s+1)}{2}+i_1+\cdots+i_s}\\
\times \beta^{1-\epsilon_1}P^{i_1}(x_1y_1\otimes (\cdots(\beta^{1-\epsilon_s}P^{i_s}(x_sy_s^{-1}\otimes m))\cdots)).
\end{multline*}

Since, under the isomorphism $\Sigma^{-s}(\Gamma^+M)_s\to (\Gamma^+\Sigma^{-s}M)_s$, the image of the element $\Sigma^{-s}u_1^{\epsilon_1}v_1^{(p-1)i_1-\epsilon_1}\cdots u_s^{\epsilon_s}v_s^{(p-1)i_s-\epsilon_s}S_s(m)$ is equal to
\[
(-1)^{s(\deg\gamma+\deg m)}u_1^{\epsilon_1}v_1^{(p-1)i_1-\epsilon_1}\cdots u_s^{\epsilon_s}v_s^{(p-1)i_s-\epsilon_s}S_s(\Sigma^{-s}m),
\]
we obtain the assertion of the proposition.

The proof is complete.
\end{proof}
Given an unstable $A$-module $M$, let, for each $s\geq0$,
\[
j_s^M:=(\psi_s^M)^\#\circ(\varphi_s^M)^\#:\fp\tsor{A}\mathscr{R}_sM\to \fp\tsor{A}(P_s\otimes M).
\]

The following corollary is followed from Theorem \ref{thm:representation of varphi_s in Gamma} and Proposition \ref{pro:transfer in Gamma}.
\begin{corollary}\label{cor:representation of j_s}
The map $\mathscr{R}_sM\to P_s\otimes M$ give by $\gamma\mapsto-\gamma$ induces the homomorphism $j_s^M$.
\end{corollary}
\section{Changes needed if $p=2$}
Our framework is also valid for $p=2$ with a suitable modification.

For $p=2$, the lambda algebra is generated by $\lambda_i$ of degree $i$ for $i\geq0$ satisfying the adem relations:
\[
\lambda_i\lambda_j=\sum_{t}\binom{t-j-1}{2t-i}\lambda_{i+j-t}\lambda_t,\quad i>2j.
\]
Therefore, a monomial $\lambda_I=\lambda_{i_1}\cdots\lambda_{i_s}\in\Lambda_s$ is called admissible if $i_k\leq 2i_{k+1}$ for $1\leq k<s$.
The excess of $\lambda_I$ or $I$ is defined by
\[
e(\lambda_I)=e(I)=i_1-i_2-\cdots-i_s.
\] 

The Dyer-Lashof algebra is the quotient algebra of $\Lambda$ by the (two sides) ideal of $\Lambda$ generated by all monomials of negative excess. We also denote $Q^i$ the image of $\lambda_i$ under the canonical projection.

Hence, Proposition \ref{pro:dual of Singer functor} and Proposition \ref{pro:representation of vaphi_s in lambda} become respectively as follows.
\begin{proposition}\label{pro:dual of Singer functor p=2}
Given an unstable $A$-module $M$, the set 
\[
\mathscr{S}=\left\{Q^I\otimes \ell:\ell\in M^\#,I\text{ admissible and }e(I)\geq |\ell|\right\}
\]
 represents an $\f2$-basis of $(\mathscr{R}_sM)^\#$.
\end{proposition}
\begin{proposition}\label{pro:representation of varphi in lambda p=2}
For any unstable $A$-module $M$, the projection $\widetilde{\varphi}_s^M:\Lambda_s\otimes M^\#\to (\mathscr{R}_sM)^\#$ given by
\[
\lambda_I\otimes \ell\to[Q^I\otimes \ell]
\]
is a chain-level representation of the mod $2$ Lannes-Zarati homomorphism $\varphi_s^M$.
\end{proposition}

For $M=\f2$ and $M=\widetilde{H}^*(B\mathbb{Z}/2)$, the squaring operations $\widetilde{Sp}^0$s acting on $\Lambda\otimes M^\#$ induce squaring the operations acting on $(\mathscr{R}_sM)^\#$. Moreover, these squaring operations induce squaring operations $Sq^0$s acting on the domain and the range of the mod $2$ Lannes-Zarati homomorphism.
\begin{proposition}\label{pro:squaring}
The squaring operations $Sq^0$s commute with each other through the Lannes-Zarati homomorphism. In other words, the following diagram is commutative
\[
\xymatrix{
\Ext{A}{s,s+t}{M}{\f2}\ar[r]^{Sq^0}\ar[d]_{\varphi_s^M}&\Ext{A}{s,2(s+t)}{M}{\f2}\ar[d]_{\varphi_s^M}\\
(\f2\tsor{A}\mathscr{R}_sM)^\#_t\ar[r]^{Sq^0}&(\f2\tsor{A}\mathscr{R}_sM)^\#_{2t+s},
}
\]
for $M=\f2$ and $M=\widetilde{H}^*(B\mathbb{Z}/2)$.
\end{proposition}
The rest of the section recovers all known results for $p=2$.
\begin{proposition}[\cite{Lan-Zar87},  \cite{Hung.Peterson1995}, \cite{Hung97}, \cite{Hung2003}, \cite{Hung.et.al2014}]
\begin{enumerate}
\item The first Lannes-Zarati homomorphism $\varphi_1^{\f2}$ is an isomorphism.
\item The second Lannes-Zarati homomorphism $\varphi_2^{\f2}$ is an epimorphism.
\item The $s$-th Lannes-Zarati homomorphism $\varphi_s^{\f2}$ vanishes at all positive stems in $\Ext{A}{s}{\f2}{\f2}$ for $3\leq s\leq 5$.
\end{enumerate}
\end{proposition}
\begin{proof}
The statements (1) and (2) are easily proved by using Proposition \ref{pro:representation of varphi in lambda p=2} and the representations of $h_i$ and $h_ih_j$ on lambda algebra (see Lin \cite{Lin08} for example).

We only give an illustrated example for our method, the detail proof of (3) is computed by the same argument.

We will to prove $\varphi_5^{\f2}(U_i)=0$. By Chen \cite{Chen2011}, the element $U_i\in\Ext{A}{5,2^{i+8}+2^{i+3}+2^i}{\f2}{\f2}$ is represented in lambda algebra by the cycle
\[
\widetilde{U}_i=(\widetilde{Sq}^0)^i(\lambda_{191}(\lambda_{15}^2\lambda_{39}+\lambda_{39}\lambda_{15}^2)\lambda_0+\lambda_{63}^2\lambda_{47}\lambda_{87}\lambda_0+\lambda_{127}\lambda_{31}\lambda_{63}\lambda_{39}\lambda_0), i\geq0.
\]
Since $\widetilde{\varphi}_4^{\f2}(\lambda_{31}\lambda_{63}\lambda_{39}\lambda_0)=0$, $\widetilde{\varphi}_3^{\f2}(\lambda_{47}\lambda_{87}\lambda_0)=0$ and $\widetilde{\varphi}_4^{\f2}(\lambda_{15}^2\lambda_{39}\lambda_0)=0$, then
\[
\widetilde{\varphi}_5^{\f2}(\widetilde{U}_0)=Q^{191}Q^{39}Q^{15}Q^{15}Q^0.
\]
Applying the adem relation, we get $Q^{15}Q^0=0\in R_2$, it implies that $\widetilde{\varphi}_5^{\f2}(\widetilde{U}_0)=0$. Hence, ${\varphi}_5^{\f2}({U}_0)=0$ and then $\varphi_5^{\f2}(U_i)=\varphi_5^{\f2}((Sq^0)^i(U_0))=(Sq^0)^i(\varphi_5^{\f2}(U_0))=0$.
\end{proof}
\begin{proposition}[\cite{Hung-Tuan-preprint}]
\begin{enumerate}
\item The zero-th Lannes-Zarati homomorphism $\varphi_0^P$ is an isomorphism on $\Ext{A}{0}{P}{\f2}$.
\item The first Lannes-Zarati homomorphism $\varphi_1^P$ is a monomorphism on ${\rm Span}\{h_i\widehat{h}_j:i\geq j\}$ and vanishes on ${\rm Span}\{h_i\widehat{h}_j:i< j\}$.
\item The $s$-th Lannes-Zarati homomorphism $\varphi_s^P$ vanishes in all positive stems in $\Ext{A}{s}{P}{\f2}$ for $2\leq s\leq 4$.
\end{enumerate}
\end{proposition}
\begin{proof}
The statements (1) and (2) are easily proved by using Proposition \ref{pro:representation of varphi in lambda p=2} and the representations of $\widehat{h}_i$ and $h_i\widehat{h}_j$ on $\Lambda\otimes P^\#$ (see Lin \cite{Lin08} for example).

Similar to above proposition, here we only give some illustrated examples, the detail proof of (3) is followed by the same argument.

First, we prove $\varphi_2^P(\widehat{c}_i)=0$ for $i\geq0$. From Lin \cite{Lin08}, $\widehat{c}_i\in\Ext{A}{2,2^{i+3}+2^{i+1}+2^i-1}{P}{\f2}$ is represented in $\Lambda\otimes P^\#$ by the cycle
\[
\bar{c}_i=(\widetilde{Sq}^0)^i(\lambda_3^2b^{[2]}), i\geq0.
\]
Since $e(\lambda_3^2)=0<2$, then, using Proposition \ref{pro:representation of varphi in lambda p=2} and Proposition \ref{pro:dual of Singer functor p=2}, it follows that $\widetilde{\varphi}_2^P(\bar{c}_0)=0$, hence, $\varphi_2^P(c_0)=0$. Therefore, $\varphi_2^P(c_i)=(Sq^0)^i(\varphi_2^P(c_0))=0$ for $i\geq0$.

Second, we will show that $\varphi_3^P(\alpha_{16}(i))=0$ for $i\geq0$. From Lin \cite{Lin08}, the element $\alpha_{16}(i)\in\Ext{A}{3,2^{i+4}+2^{i+2}-1}{P}{\fp}$ is represented in $\Lambda\otimes P^\#$ by the cycle
\[
\bar{\alpha}_{16}(i)=(\widetilde{Sq}^0)^i(\lambda_7^2\lambda_0b^{[2]}+(\lambda_3^2\lambda_9+\lambda_7\lambda_5\lambda_3)b^{[1]}),i\geq0.
\]
Using Proposition \ref{pro:representation of varphi in lambda p=2} and Proposition \ref{pro:dual of Singer functor p=2}, it is easy to verify that $\widetilde{\varphi}_3^P(\bar{\alpha}_{16}(0))=Q^7Q^7Q^0b^{[2]}$. 

Since $e(Q^7Q^7Q^0)=0<2$, it implies that $\widetilde{\varphi}_3^P(\bar{\alpha}_{16}(0))=0$. Hence, ${\varphi}_3^P({\alpha}_{16}(0))=0$ and then ${\varphi}_3^P({\alpha}_{16}(i))=0$ for $i\geq0$.

Finally, we will verify that $\varphi_4^P(\gamma_{63}(i))=0$ for $i\geq0$. From Lin \cite{Lin08}, the element $\gamma_{63}(i)\in \Ext{A}{4,2^{i+6}+2^{i+1}+2^i-1}{P}{\f2}$ is represented in $\Lambda\otimes P^\#$ by the cycle
\begin{multline*}
\bar{\gamma}_{63}(i)=(\widetilde{Sq}^0)^i(\lambda_{31}\lambda_7\lambda_{23}\lambda_0b^{[2]}+\lambda_{47}(\lambda_3^2\lambda_9+\lambda_9\lambda_3^2)b^{[1]}\\
+(\lambda_{15}^2\lambda_{11}\lambda_{21}+\lambda_{31}\lambda_7\lambda_{15}\lambda_9+\lambda_{15}\lambda_{47}\lambda_0^2)b^{[1]}),i\geq0.
\end{multline*}

By the same method, it is easy to verify that $\widetilde{\varphi}_4^P(\bar{\gamma}_{63}(0))=Q^{47}Q^9Q^3Q^3b^{[1]}$. 

Applying the adem relation, we get that $Q^9Q^3Q^{3}=0\in R_3$, then $\widetilde{\varphi}_4^P(\bar{\gamma}_{63}(0))=0$. It implies that ${\varphi}_4^P({\gamma}_{63}(0))$ and then ${\varphi}_4^P({\gamma}_{63}(i))=0$ for $i\geq0$.
\end{proof}
\begin{remark}\label{rm:vaphi_1^P for p=2}\rm
Basing on Proposition \ref{pro:dual of Singer functor p=2}, it is easy to verify that $(\f2\tsor{A}\mathscr{R}_1P)^\#$ is spanned by
\[
\left\{\left[Q^{2^i-1}b^{[2^j-1]}\right]: i\geq j\right\}\cup\left\{(Sq^0)^i\left(\left[Q^{2(2^{j}-1)}b^{[1]}\right]+\left[Q^{2^{j+1}-1}b^{[2]}\right]\right):i\geq0, j\geq1\right\}.
\]

Therefore, the first Lannes-Zarati homomorphism $\varphi_1^P$ is not an epimorphism.
\end{remark}
\providecommand{\bysame}{\leavevmode\hbox to3em{\hrulefill}\thinspace}
\providecommand{\MR}{\relax\ifhmode\unskip\space\fi MR }
\providecommand{\MRhref}[2]{%
  \href{http://www.ams.org/mathscinet-getitem?mr=#1}{#2}
}
\providecommand{\href}[2]{#2}

\end{document}